\documentclass[]{article}
\usepackage{agor}
\usepackage{graphicx}
\usepackage{amsmath,amssymb}
\usepackage{textcomp}
\usepackage[numbers,sort]{natbib}
\newcommand{\ra}[1]{\renewcommand{\arraystretch}{#1}}

\usepackage{multirow}
\DeclareMathOperator{\Conv}{Conv}
\DeclareMathOperator{\clconv}{cl \ conv}
\DeclareMathOperator{\dom}{dom}

\DeclareMathOperator{\proj}{proj}

\usepackage{hyperref}
\usepackage{caption,color}
\usepackage{booktabs}

\usepackage{subfig}

\providecommand{\revised}[1]{{\protect\color{black}{#1}}}
\providecommand{\rev}[1]{{\protect\color{black}{#1}}}

\newcommand{\ignore}[1]{}

\begin{document}

\title{Ideal formulations for  constrained convex optimization problems with indicator variables
}

\author{Linchuan Wei\thanks{Department of Industrial Engineering and Management Sciences,
	Northwestern University,
	Evanston, IL, USA. 
		Email: \texttt{LinchuanWei2022@u.northwestern.edu}} \and
        Andr\'es G\'omez\thanks{Department of Industrial and Systems Engineering, University of Southern California, Los Angeles, CA, USA.
        	Email: \texttt{gomezand@usc.edu}} \and
        Simge K{\"u}{\c{c}}{\"u}kyavuz\thanks{Department of Industrial Engineering and Management Sciences,
        	Northwestern University,
        	Evanston, IL, USA. 
        	Email: \texttt{simge@northwestern.edu}}
}

\maketitle

\begin{abstract}
Motivated by modern regression applications, in this paper, we study the convexification of a class of convex optimization problems with indicator variables and combinatorial constraints on the indicators. Unlike most of the previous work on convexification of sparse regression problems, we simultaneously consider the nonlinear non-separable objective, indicator variables, and combinatorial constraints. Specifically, we give the convex hull description of the epigraph of the composition of a one-dimensional convex function and an affine function under arbitrary combinatorial constraints. As special cases of this result, we derive ideal convexifications for problems with hierarchy, multi-collinearity, and sparsity constraints. Moreover, we also give a short proof that for a separable objective function, the perspective reformulation is ideal independent from the  constraints of the problem.  Our computational experiments with \revised {sparse} regression problems demonstrate the potential of the proposed approach in improving the relaxation quality without significant computational overhead.


\end{abstract}

\section{Introduction}
Given a set $Q\subseteq \{0,1\}^p$, a vector $h\in \R^p$ such that $h_i\ne 0$, for  all $i \in [p]\revised{:=\left\{1,\ldots,p\right\}}$, and a convex function $f:\R\to \R$, we study the set 
\begin{equation*}
Z_Q = \left\{(z, \beta, t) \in Q \times \mathbb R^{p} \times  \mathbb R ~ | ~  f\left(h^\top\beta\right) \leq t, \beta_i(1 - z_i) = 0, \forall i \in [p]\right\}.\end{equation*}
In set $Z_Q$ above, $z$ is a vector of indicator variables with $z_i=1$ if $\beta_i\neq 0$, and the set $Q$ encodes combinatorial constraints on the indicator variables. We assume without loss of generality that $f(0)=0$, since this assumption can always be satisfied after subtracting the constant term~$f(0)$.

The motivation to study $Z_Q$ stems from sparse regression problem:
Given a set of observations $(x_i,y_i)_{i=1}^n$ where $x_i\in \R^p$ are the features corresponding to observation $i$ and $y_i \in \R$ is its associated response variable, inference with a sparse linear model can be modeled as the optimization problem
\begin{subequations}\label{eq:constrainedRegression}
	\begin{align}
	\min_{z,\beta}\;& \sum_{i=1}^n f\left(y_i , x_i^\top \beta\right) +\lambda \rho(\beta)\label{eq:reg-obj}\\
	\text{s.t.}\;&\beta_i(1-z_i)=0,&i\in [p] \label{eq:reg-comp}\\
	& \beta\in \R^p,\; z\in Q\subseteq \{0,1\}^p, \label{eq:reg-comb}
	\end{align}
\end{subequations}
where $\beta$ is a vector of regression coefficients, $f$ is a loss function, $\lambda\geq0$ is a regularization parameter and $\rho$ is regularization function. Often, $f(\beta)=\left(y_i - x_i^\top \beta\right)^2$, in which case \eqref{eq:constrainedRegression} is referred to as sparse least squares regression, and typical choices of $\rho$ include $\ell_0$, $\ell_1$\revised{,} or $\ell_2$ regularizations. 

If $Q$ is defined via a $q$-sparsity constraint, $Q=\left\{z\in \{0,1\}^p~|~\sum_{i=1}^p z_i\leq q\right\}$, then problem \eqref{eq:constrainedRegression} reduces to the best subset selection problem \cite{miller2002subset}, a fundamental problem in statistics. Nonetheless, constraints other than the cardinality constraint arise in several statistical problems. \citet{bertsimas2016} suggest imposing constraints of the form $\sum_{i\in S}z_i\leq 1$ for some $S\subseteq [p]$ to prevent multicollinearity; \citet{carrizosa2020linear} use similar constraints to capture nested categorical variables. Constraints of the form $z_i\leq z_j$ can be used to impose strong hierarchy relationships, and constraints of the form $z_i\leq \sum_{j\in H\subseteq [p]}z_j$ can be used for weak hierarchy relationships \cite{bien2013lasso}. In group variable selection, indicator variables of regression coefficients of variables in the same group are linked, see \cite{huang2012selective}. \citet{MKS19} and \citet{KSMW20} impose that the indicator variables, which correspond to edges in an underlying graph, do not define cycles---a necessary constraint for inference problems with causal graphs. \citet{cozad2015combined} suggest imposing a variety of constraints in both the continuous and discrete variables to enforce priors from human experts. 

Problem \eqref{eq:constrainedRegression} is $\mathcal{NP}$-hard even for a $q$-sparsity constraint \cite{Natarajan1995}, and is often approximated with a convex surrogate such as lasso \cite{hastie2015statistical,tibshirani1996regression}. Solutions with better statistical properties than lasso can be obtained from non-convex continuous approximations \cite{fan2001variable,zhang2010nearly}. Alternatively, it is possible to solve \eqref{eq:constrainedRegression} to optimality via branch-and-bound methods \cite{bertsimas2016best,cozad2014learning}. In all cases, most of the approaches for \eqref{eq:constrainedRegression} have focused on the $q$-sparsity constraint (or its Lagrangian relaxation). For example, a standard technique to improve the relaxations of \eqref{eq:constrainedRegression} revolves around the use of the \emph{perspective reformulation}
\cite{akturk2009strong,ceria1999convex,Stubbs1999,dong2013valid,dong2015regularization,frangioni2006perspective,frangioni2007sdp,frangioni2011projected,frangioni2016approximated,gunluk2010perspective,hijazi2012mixed,wu2017quadratic,zheng2014improving}, an ideal formulation of a separable quadratic function with indicators (but no additional constraints). Recent work on obtaining ideal formulations for non-separable quadratic functions \cite{atamturk2018strong,atamturk2018sparse,atamturk2019rank,dong2013valid,frangioni2019decompositions,jeon2017quadratic} also ignores additional constraints in $Q$. 

There is a recent research thrust on studying constrained versions of \eqref{eq:constrainedRegression}. \citet{dong2019structural} study problem \eqref{eq:constrainedRegression} from a continuous optimization perspective (after projecting out the discrete variables), see also \cite{dong2019integer}. \citet{hazimeh2019learning} give specialized algorithms for the natural convex relaxation of~\eqref{eq:constrainedRegression} where $Q$ is defined via strong hierarchy constraints. Several results exist concerning the convexification of nonlinear optimization problems with constraints \cite{anstreicher2012convex,belotti2015conic,bienstock2014cutting,burer2017convexify,modaresi2016intersection,richard2010lifting,vielma2019small,kilincc2014two,wang2019tightness,wang2021tightness,wang2019generalized,Burer2009}, 
but such methods in general do not deliver ideal, compact or closed-form formulations for the specific case of problem \eqref{eq:constrainedRegression} with structured feasible regions. In a recent work closely related to the setting considered here, \citet{xie2018ccp} prove that the perspective formulation is \emph{ideal} if the objective is quadratic and separable, and $Q$ is defined by a $q$-sparsity constraint. In a similar vein, \citet{baccinew2019} show that the perspective reformulations for convex differentiable functions are tight for 1-sum compositions, and they use this result to show that they are ideal under unit commitment constraints. However, similar results for more general (non-separable) objective functions or constraints are currently not known.

\textbf{Our contributions.} In this paper, we provide a first study (from a convexification perspective) of the interplay between \emph{non-separable} convex  objectives and combinatorial constraints on the indicator variables. Specifically, we derive the convex hull description of $Z_Q$: the result is stated in terms of the convexification of the combinatorial set $Q$, but places no  assumptions on its form.  Using this result,  we develop  ideal formulations for settings in which the logical constraints on the indicator variables encode sparsity constraints or the so-called strong and weak hierarchy relations. In addition, we generalize the result in \cite{xie2018ccp} and \cite{baccinew2019} to arbitrary constraints on $z$ for \emph{separable} convex functions $f$, \revised{in our setting}. We show the computational benefit of the proposed approach on constrained regression problems with hierarchical relations. 

An earlier version of this work appeared in 
\cite{WGK20}, where we only considered separable and rank-one convex \emph{quadratic} functions, and sparsity and strong hierarchy constraints. Furthermore, in \cite{WGK20}, our proofs of the convexification results use the structure of each of the sets considered, whereas in the present paper, we give a unifying technique  that generalizes to any combinatorial set for functions that are not necessarily quadratic. Finally, here, we expand on our preliminary computational experiments in \cite{WGK20} with additional datasets, conduct a further analysis on the choices of the regularization parameters\revised{, and perform computations with sparse logistic regression}.

\textbf{Notation.} Given a  one-dimensional convex function $f:\R\to \R$, we adopt the convention that $0f(\beta/0)=\lim_{z\to 0^+}zf(\beta/z)$. Using this convention, the function $zf(\beta/z)$ for $z\geq 0$ is the closure of the perspective function of $f$, and is convex. Let $\mathbf 0$ and $\mathbf 1$ be  vectors of conformable dimension with all zeros and ones, respectively, and let $e_i$ denote the $i$th unit vector of appropriate dimension with 1 in the $i$th component and zeros elsewhere. For a set $Q$, we denote by $\conv(Q)$ its convex hull and by $\clconv(Q)$ the closure of its convex hull. Given two vectors $u,v$ of same dimensions, we let $u\circ v$ denote the Hadamard vector of $u$ and $v$, i.e., $(u\circ v)_i=u_iv_i$.

\section{Convexification of $Z_Q$}\label{sec:convexification}
Observe that in set $Z_Q$, the coefficients of $\beta$ can be scaled and negated if necessary to ensure $h_i=1$ for all $i\in [p]$. 
Therefore, in the derivation of ideal formulations in this section, we assume, without loss of generality, that 
\begin{equation*}
Z_Q = \left\{(z, \beta, t) \in Q \times \mathbb R^{p} \times  \mathbb R ~ | ~  f\left(\mathbf 1^\top  \beta\right) \leq t, \beta_i(1 - z_i) = 0, \forall i \in [p]\right\}.\end{equation*}
We also assume, without loss of generality, that for every $i\in [p]$ there exists $z\in Q$ such that $z_i=1$, as otherwise $z_i=\beta_i=0$ can be fixed and the corresponding variables can be removed.

For a given set $Q$, let $Q^0 = Q \setminus \{\mathbf 0\}$ or, equivalently, $Q^0 = \{ z \in Q~|~ \sum_{i=1}^p z_i \geq 1\}$. As we show in the subsequent discussion, the convexification of the set $Z_Q$ relies on the characterization of $\conv(Q^0)$. To this end, we first establish such a characterization. 

\begin{proposition}\label{prop:F}
The convex hull of $Q^0$ admits a description as \begin{equation}\label{eq:Q0}\conv(Q^0) = \conv(Q) \bigcap \{ z \; | \; \pi^\top z \geq 1, \; \forall\pi \in  \mathcal F\},\end{equation}
where $\mathcal F$ is a finite subset of $\R^p$. 
\end{proposition}
\
\begin{proof}
	Let $\pi^\top z\geq \pi_0$ be an arbitrary valid inequality for $\conv(Q^0)$. If $\pi_0>0$, then $\frac{1}{\pi_0}\pi^\top z\geq 1$ is an equivalent inequality satisfying the conditions in \eqref{eq:Q0}. Otherwise, if $\pi_0\leq 0$, then the inequality does not cut off  $\mathbf 0$ and is thus valid for $Q$ and $\conv(Q)$. Therefore, it follows that $\conv(Q)\subseteq \{ z \; | \; \pi^\top z \geq \pi_0\}$, and inequality $\pi^\top  z\geq \pi_0$ is either already a facet of $\conv(Q)$, or is implied by the facets $\conv(Q)$. Finally, finiteness of $\mathcal F$ follows since $\conv(Q^0)$ is a polyhedron.~ 
\end{proof}
Note that if $\mathbf{0}\not\in Q$, then  $\mathcal F=\emptyset$. In practice, a set $\mathcal F$ of minimal cardinality is preferred. \revised{Since $\conv(Q)$ and $\conv(Q_0)$ may have an exponential number of facets, set $\mathcal F$ may be exponentially large as well.} \revised{In such cases, inequalities from $\mathcal F$ can be generated if violated in an iterative fashion, as is standard in a cutting plane algorithm.} \revised{Note that even if $\conv(Q)$ is simple, $\conv(Q_0)$ may contain an exponential number of facets. Nonetheless, in such cases, $\conv(Q_0)$ admits a compact extended formulation \cite{forbidden2015}, which in turn implies that separation of the inequalities in $\mathcal F$ can be done in polynomial time. }

Intuitively, one may think of $\mathcal F$ as the set of ``new" facets of $\conv(Q^0)$ that are not facets of $\conv(Q)$. If $\conv(Q)$ and $\conv(Q^0)$ have the same dimension, this intuition is correct. However, if the dimension of $\conv(Q^0)$ is less than the dimension of $\conv(Q)$, it may be the case that $\conv(Q^0)\subseteq \{z:\pi^\top z=1\}$ for some $\pi\in \mathcal F$, and thus this inequality is not a facet. For example, if $Q=\{0,1\}$, then $\conv(Q)=[0,1]$, $Q^0=\conv(Q^0)=\{1\}$ and $\mathcal F=\{1\}$, but the inequality $z\geq 1$ is not a facet of  the 0-dimensional polyhedron $\conv(Q^0)$.  

The description of $\clconv(Z_Q)$ depends on the structure of $Q$, and is critically dependent on whether the variables can be partitioned into multiple mutually exclusive components. We formalize this characteristic next.    

\begin{definition}\label{def:G}
	For $i,j\in[p], i\ne j$, define $i \sim j$ if there exists some $z \in Q$ such that $z_i=z_j =1$. Define the graph $G_Q =(V,E)$ where  $V=[p]$ and $\{i,j\}\in E$ if and only if $i\sim j$.  
\end{definition}

\subsection{The connected case}\label{sec:connected}

In this section, we provide ideal formulations in the original space of variables when graph $G_Q$ in Definition~\ref{def:G} is connected. This assumption is satisfied in most of the practical applications we consider, see \S\ref{sec:cases}. Later, in \S\ref{sec:disconnected}, we build upon the results of this section to derive ideal formulations when $G_Q$ is not necessarily connected.

Before we propose a class of valid inequalities for $Z_Q$, we give a lemma.

\begin{lemma}\label{lemma:perspective}
	For a one-dimensional proper convex function $f: \R \rightarrow \R$ with effective domain $\dom(f) = \R$, $f(0) = 0$ and its perspective $g(x,t) = t f(\frac{x}{t}) \; : \R^2 \rightarrow \R$, if $0 < t_1 \leq t_2$, then $g(x, t_1) \geq g(x, t_2)$ for all $x \in \R$. 
\end{lemma}

\begin{proof}
It suffices to show that the function $\phi(x) = g(x, t_1) - g(x, t_2)$ is non-decreasing in $[0, +\infty]$ and non-increasing in $[-\infty , 0]$. Since $\dom(f) = \R$, $f$ is continuous over $\R$ so is $\phi(x)$. Also, by convexity, we know that the right-derivative of $f(x)$ exists and is non-decreasing. Thus, $\phi_{+}^{'}(x) = f_{+}^{'}(\frac{x}{t_1}) - f_{+}^{'}(\frac{x}{t_2}) \geq 0 $ for all $x \in [0, +\infty]$. A continuous function with non-negative right-derivative is non-decreasing \cite{hardy2018course}. For $x \in [-\infty, 0]$, the left-derivative of $\phi$ is $\phi_{-}^{'}(x) = f_{-}^{'}(\frac{x}{t_1}) - f_{-}^{'}(\frac{x}{t_2}) \leq  0$, and similarly, $\phi(x)$ is non-increasing in $[-\infty , 0]$.
 
\end{proof}

\begin{proposition}
 The inequalities 
\begin{equation}\label{eq:valid}t \geq (\pi^\top z)f\left(\frac{\mathbf 1^\top  \beta}{\pi^\top z}\right),\; \; \forall \pi \in \mathcal F\end{equation}
are valid for $Z_Q$ for any finite set $\mathcal F\subseteq \R^p$ satisfying \eqref{eq:Q0}.   
\end{proposition}

\begin{proof}
First, observe that if  $\mathbf 0 \not \in Q$, then $\mathcal F=\emptyset$ and the statement is superfluous. Suppose, $\mathcal F\ne \emptyset.$ 
We consider two cases. If $z\ne \mathbf 0$, then we have $\pi^\top z\ge 1$ for $\pi\in \mathcal F$. Then, from Lemma~\ref{lemma:perspective},  $(\pi^\top z)f\left(\frac{\mathbf 1^\top  \beta}{\pi^\top z}\right)\le f\left(\mathbf 1^\top  \beta\right)\le t$. Hence the inequality is valid.  
Finally, if $z = \mathbf 0$, then $\beta=\mathbf 0$ in $Z_Q$. Therefore, $$t\ge f\left(\mathbf 1^\top  \beta\right) = f(0)=0=\lim_{\zeta\to 0^+}\zeta f\left(0/\zeta\right)=(\pi^\top z)f\left(\frac{\mathbf 1^\top  \beta}{\pi^\top z}\right),$$ and the inequality is valid.  
\end{proof}

We now describe the closure of the convex hull of $Z_Q$ under the assumption that graph $G_Q$ described in Definition~\ref{def:G} is connected. 
\begin{theorem}\label{theo:hullConnected}
	If the graph $G_Q$ given in Definition \ref{def:G} is connected, then 
	\begin{align}\clconv(Z_Q) = \Big\{(z,\beta,t) \in [0,1]^p \times \R^{p}&\times \R \; | \; z \in \conv(Q),\;  t \geq f(\mathbf 1^\top  \beta),\nonumber\\
	  &t \geq (\pi^\top z)f\left(\frac{\mathbf 1^\top  \beta}{\pi^\top z}\right),\; \; \forall \pi \in \mathcal F \Big\}\label{eq:defY}
	\end{align}
	for any finite set $\mathcal F\subseteq \R^p$ satisfying \eqref{eq:Q0}. 
\end{theorem}

Note that if $\mathbf 0\not\in Q$, i.e., $\mathcal F=\emptyset$, then Theorem~\ref{theo:hullConnected} states that the description of $\clconv(Z_Q)$ is obtained simply by dropping the complementarity constraints $ \beta_i(1 - z_i) = 0, \forall i \in [p]$ and independently taking the convex hull of $Q$. Otherwise, since the description of $\clconv(Z_Q)$ requires a new inequality for every element of $\mathcal F$, a minimal description of $\mathcal F$ is certainly preferred from a computational standpoint. If $\conv(Q^0)$ is full-dimensional, the strongest nonlinear inequalities \eqref{eq:valid} are obtained from facets of $\conv(Q^0)$. Moreover, in many situations, it may not be possible to have a full description of $\conv(Q)$ or $\conv(Q^0)$; nonetheless, in those cases, it may be possible to obtain a facet $\bar \pi^\top x\geq 1$ of $\conv(Q^0)$, and Theorem~\ref{theo:hullConnected} ascertains that the valid inequality 
\begin{equation}t \geq (\bar \pi^\top z)f\left(\frac{\mathbf 1^\top  \beta}{\bar \pi^\top z}\right)\end{equation} is not dominated by any other inequality of a similar form, and that inequalities of this form are sufficient to describe $\clconv(Z_Q)$. \revised{In Appendix \ref{sec:app} we focus on the special case where $\conv(Q)$ admits a compact representation but $\conv(Q_0)$ has exponentially many inequalities: We show how to use a compact extended formulation of $\conv(Q_0)$ to derive the description of $\clconv(Z_Q)$ in a higher dimensional space.}

 Before proving Theorem~\ref{theo:hullConnected}, we give a lemma used in the proof.

\begin{lemma}\label{lemma:Z0}
	$z \in \conv(Q)$ if and only if there exists some $\alpha \in [0,1]$ and $z^0 \in \conv(Q^0)$ such that $z = \alpha z^0$.
\end{lemma}
\begin{proof}
Note that if $\mathbf 0\not\in Q$, then the result holds trivially by letting $\alpha=1$. Therefore, we will assume that $\mathbf 0\in Q$.

	($\Rightarrow$)  Let  $z \in \conv(Q)$. So we can write  $z$  as a convex combination of the extreme points of $Q$. Specifically, we distinguish between  the feasible points $z^i\in Q^0$ for $i\in\mathcal I$ and the origin. In particular, there exists $\gamma\ge \mathbf 0$ with $\sum_{i \in \mathcal I\cup\{0\}}\gamma_i=1$, such that
	\begin{align*}
	z &= \gamma_0 \mathbf{0} + \sum_{i \in \mathcal I}\gamma_i z^i  = \left(\sum_{i \in \mathcal I}\gamma_i\right) \sum_{i \in \mathcal I} \frac{\gamma_i}{\sum_{i \in \mathcal I} \gamma_i} z^i.
	\end{align*}   
	Letting $\alpha=\sum_{i \in \mathcal I}\gamma_i$, the result follows. 
	
	($\Leftarrow$)  Let $z = \alpha z^0$ for some $\alpha \in [0,1]$ and $z^0 \in \conv(Q^0)$; by definition, we can expand $z^{0}$ as 
	$z^{0} = \sum_{i \in \mathcal I} \gamma_i z^{i}$, a convex combination of $z^{i} \in Q^0$. By adding the term $(1 - \alpha) \mathbf{0}$, we have
	$z = (1 - \alpha) \mathbf{0} + \sum_{i \in \mathcal I} \alpha \gamma_i z^{i}$.~ 
\end{proof}

We are now ready to prove Theorem~\ref{theo:hullConnected}.
\begin{proof}[Proof of Theorem~\ref{theo:hullConnected}]
	Define $Y$ as the set described by \eqref{eq:defY}. Let $a,b\in \R^p, c\in \R$, and consider the two optimization problems
	\begin{align}
	\min_{z, \beta, t}\;& a^\top z + b^\top \beta  + c t \quad \text{subject to} \quad  (z,\beta,t) \in Z_Q,\text{ and}\label{eq:zbtinZQ_G}\\
	\min_{z, \beta, t}\;& a^\top z + b^\top \beta  + c t \quad \text{subject to} \quad  (z,\beta,t) \in Y.\label{eq:zbtinY_G}
	\end{align}
	We show that there exists a solution $(z,\beta,t)$ optimal for both problems, and that the corresponding objective values of both problems coincide.

	\paragraph{$\bullet$ Simple  cases:}
	 
	If $c < 0$, then both \eqref{eq:zbtinZQ_G} and \eqref{eq:zbtinY_G} are unbounded. To see this, let  $z= \beta = \mathbf 0$, and $t = \kappa$, where $\kappa\ge 0$. This solution is feasible for both \eqref{eq:zbtinZQ_G} and \eqref{eq:zbtinY_G}. Letting $\kappa\to  \infty$, the objective goes to minus infinity. 
	 
	If $c = 0$ and $b \neq \mathbf 0$, then let $z_j=1$ for some $j\in[p]$ such that $b_j\ne 0$, and let $\beta_j $ go to plus or minus infinity depending on whether $b_j$ is negative or positive, respectively, while keeping $\beta_i=0$ for $i\ne j$. Again, the objective goes to minus infinity.
	
	If $c = 0$ and $b = \mathbf 0$, then these two problems reduce to minimizing $a^\top z$ over $\conv(Q)$ and thus \eqref{eq:zbtinZQ_G} and \eqref{eq:zbtinY_G} are equivalent. 
	
	If $c > 0$, then we assume, without loss of generality, that $c = 1$ by scaling. If there exists $i_0 \neq j_0$ such that $b_{i_0} \neq b_{j_0}$, then there exists some $i$ and $j$ in a path from $i_0$ and $j_0$ in $G_Q$ such that $i\sim j$ and $b_i\neq b_j$, and without loss of generality, we assume $b_i<b_j$. Furthermore, there exists some $z \in Q$ such that $z_i = z_j = 1$.
	Then we take such a vector $z$, we let $\beta$ be a vector of zeros except for $ \beta_i = - \beta_j=\kappa$ for some $\kappa>0$, and we let $t = f(\mathbf 1^\top  \beta) = 0$. Such a triplet $(z, \beta, t)$ is in $Z_Q$ and $Y$, and by letting $\kappa\to\infty$, the objective goes to minus infinity. Therefore, we assume in the sequel that $b_i = \bar b$ for all $i\in [p]$.

	\paragraph{$\bullet$ Case $c=1$ and $b = \bar b \mathbf{1}$:} 
	
	We now show that for $b = \bar b \mathbf{1}$ problem  \eqref{eq:zbtinY_G} either has a finite optimal solution that is in set $Z_Q$ or is  unbounded. Note that \eqref{eq:zbtinY_G} is equivalent to:
	\begin{align*}
	\min_{z, \beta} \quad & a^\top z +\bar b\left(\mathbf 1^\top\beta\right) + \max\left\{f\left(\mathbf 1^\top \beta\right),\max_{\pi \in \mathcal{F}}\left\{\left(\pi^\top z\right) f\left(\frac{\mathbf 1^\top \beta}{\pi^\top z}\right) \right\}\right\}\\
	\text{s.t.} \quad & z \in \conv(Q),
	\end{align*}
	and, from Lemma~\ref{lemma:perspective}, it further simplifies to
	\begin{subequations}\label{eq:Yeq_G2}
		\begin{align}
\min_{z, \beta} \quad & a^\top z +\bar b\left(\mathbf 1^\top\beta\right) +  \min_{\pi \in \mathcal F}\{\pi^\top z,1\} f\left(\frac{\mathbf 1^\top \beta}{\min_{\pi \in \mathcal   F}\{\pi^\top z,1\}}\right)\label{eq:Yeq_G2-obj}\\
\text{s.t.} \quad & z \in \conv(Q).
\end{align}	 
\end{subequations}	
	
	Let $f^*:\R\to \R$ be the convex conjugate of function $f$, i.e., 
	$f^*(\gamma)=\sup_{x\in \R}\gamma x-f(x)$, and let $\Gamma=\left\{\gamma \in \R: f^*(\gamma)<\infty \right\}$ be the domain of $f^*$. Note that if $-\bar b\not\in \Gamma$, it follows that both \eqref{eq:zbtinZQ_G} and \eqref{eq:zbtinY_G} are unbounded. Thus, we assume in the sequel that $-\bar b\in \Gamma$.
	
	Observe that, given $w>0$, the convex conjugate of the function $wf(x/w)$ is
	 $wf^*(\gamma)$. 
	Hence, from Fenchel inequality, we find that\revised{,} for any $\beta$, $z$ such that $\pi^\top z>0$, and $\gamma\in \Gamma$,
	\begin{equation}\label{eq:fenchel}\min_{\pi \in \mathcal F}\{\pi^\top z,1\} f\left(\frac{\mathbf 1^\top \beta}{\min_{\pi \in \mathcal F}\{\pi^\top z,1\}}\right)\geq \gamma(\mathbf{1}^\top \beta)-\min_{\pi \in \mathcal F}\{\pi^\top z,1\}f^*(\gamma).\end{equation}
	Furthermore, for $\pi^\top z=0$ for some $\pi \in \mathcal F$, if the left hand side of \eqref{eq:fenchel} is infinity, then the inequality holds trivially; otherwise, if the left hand side of  \eqref{eq:fenchel} is $0f( (\mathbf 1^\top \beta)/0)=\lim_{z\to 0^+}zf( (\mathbf 1^\top \beta)/z)=d$ with $|d|<\infty$, then by continuity of the functions at both sides of the inequality, \eqref{eq:fenchel} is satisfied.
	
	Using \eqref{eq:fenchel} with $\gamma=-\bar b$ to lower bound the last term in  \eqref{eq:Yeq_G2-obj}, we obtain the relaxation
	\begin{align*}
	\min_{z, \beta, t} \quad & a^\top z + \bar b\left(\mathbf 1^\top\beta\right) +\left(-\bar b(\mathbf{1}^\top \beta)-\min_{\pi \in \mathcal{F}}\{\pi^\top z,1\}f^*(-\bar b)\right)\nonumber\\
	\text{s.t.} \quad & z \in \conv(Q),
	\end{align*}
	or, equivalently,
	\begin{subequations}\label{eq:Yeq1_G}
		\begin{align}
		\min_{z} \quad & a^\top z  +\max_{\pi \in \mathcal{F}}\{1-\pi^\top z,0\}f^*(-\bar b)-f^*(-\bar b)\\
		\text{s.t.} \quad & z \in \conv(Q).
		\end{align}
	\end{subequations}
	We will first prove that relaxation \eqref{eq:Yeq1_G} admits an optimal solution integral in $z$, and then we will show that the lower bound from the relaxation is in fact tight. 

	Note that if $\mathbf 0\not\in Q$, then $\mathcal F=\emptyset$ and there exists an optimal integer solution $z^*\in Q$ to the relaxation  \eqref{eq:Yeq1_G} with objective value $ a^\top z^* -f^*(-\bar b)$. 
	
Now consider the case that $\mathbf 0\in Q$.
	Let $z^*$ be an optimal solution of \eqref{eq:Yeq1_G}, and  consider two subcases. 
	\revised{
		\paragraph{$\bullet$ Subcase (i):}
	First, suppose that $1-\pi^\top z^*\leq 0$ for all $\pi\in \mathcal{F}$. 
	In this case, \eqref{eq:Yeq1_G} is equivalent to 
	\begin{subequations}\label{eq:case1}
		\begin{align}
		\min_{z}\quad \; & a^\top z -f^*(-\bar b)\\
		\text{s.t.}\quad \; &\pi^\top z\geq 1 &\forall \pi\in \mathcal{F}\\
		&z \in \conv(Q).
		\end{align}
	\end{subequations}
	From Proposition~\ref{prop:F}, the feasible region of \eqref{eq:case1} is precisely $\conv(Q^0)$, thus problem \eqref{eq:case1} admits an optimal integer solution $z^*\in Q^0$ with objective value $a^\top z^* -f^*(-\bar b)$. 
	
			\paragraph{$\bullet$ Subcase (ii): }
			Let $\bar \pi\in {\arg\min}_{\pi\in \mathcal{F}}\pi^\top z^*$, and suppose that $1-\bar\pi^\top z^*> 0$. In this case, problem \eqref{eq:Yeq1_G} is equivalent to 
	\begin{subequations}\label{eq:case2}
		\begin{align}
		\ell=\min_{z} \quad & a^\top z  -(\bar \pi^\top z) f^*(-\bar b)\\
		\text{s.t.} \quad& \pi^\top z\geq \bar\pi^\top z&\forall \pi\in \mathcal{F}\\
		\quad & z \in \conv(Q).
		\end{align}
	\end{subequations}
	Note that $f^*(-\bar b)=\sup_{x\in \R}-\bar bx-f(x)\geq 0$, because $x= 0$ is a possible solution to the supremum problem and $f(0)=0$. Since $z=\mathbf 0$ is feasible for \eqref{eq:case2}, we find that the objective value $\ell\leq 0$. If $\ell=0$, then $z^*=\mathbf 0$ is optimal and the proof is complete. Suppose now that $\ell<0$. Observe from Lemma~\ref{lemma:Z0} that $z^*=\alpha z_0$ for some $z_0\in \conv(Q^0)$ and $\alpha\in (0,1)$---the case $\alpha=1$ is excluded, since $1-\bar\pi^\top z_0\leq 0$ for any $z_0\in \conv(Q^0)$. Consequently, the point $\bar z=z_0=\frac{z^*}{\alpha}$, with objective value $\bar \ell=\ell/\alpha< \ell$ is feasible for \eqref{eq:case2} with better objective value than $z^*$, resulting in a contradiction.  
}

	\revised{From  subcases (i) and (ii)}, we see that either $z^*=\mathbf 0$ is feasible and optimal for relaxation \eqref{eq:Yeq1_G} (with objective value $0$), or that there exists an optimal integer solution $z^*$ with objective value $ a^\top z^* -f^*(-\bar b)$, regardless of whether $\mathbf 0\in Q$ or not. 
	We now prove that the lower bound  provided by the relaxation \eqref{eq:Yeq1_G} is tight, by finding $\beta^*\in \R^p$ such that $(z^*,\beta^*,f(\mathbf 1^\top \beta^*))$ is feasible for \eqref{eq:zbtinZQ_G} with the same objective value as  \eqref{eq:Yeq1_G}. If $z^*=\mathbf 0$, then clearly $(\mathbf 0,\mathbf 0,0)$ is optimal for \eqref{eq:zbtinZQ_G} with objective value $0$, and we now focus on the case $z^*\neq \mathbf 0$. Let $\bar x\in {\arg\sup}_{x\in \R}-\bar bx-f(x)$ and suppose that    $\bar x$ exists,  i.e., $\sup$ can be changed to $\max$,
	and observe that $f^*(-\bar b)=-\bar b\bar x-f(\bar x)$, or in other words
	$$a^\top z -f^*(-\bar b)=a^\top z +\bar b\bar x+f(\bar x).$$
	Since $z^*\neq \mathbf 0$, there exists $i$ such that $z_i^*=1$. Setting $\beta_i^*=\bar x$, $\beta_j^*=0$ for $j\neq i$, we find that the point $(z^*,\beta^*,f(\beta^*))$ is feasible for both  \eqref{eq:zbtinZQ_G} and \eqref{eq:zbtinY_G}, and since its objective value is the same as the lower bound obtained from \eqref{eq:Yeq1_G}, it is optimal for both problems. Now suppose that $\bar x$ above does not exist, but $(\bar x^1,\bar x^2,\ldots)$ is a sequence of points such that $-\bar b\bar x^i-f(\bar x^i)\to f^*(-\bar b)$. In this case, using identical arguments as above, we find a sequence of feasible points with objective value converging to $a^\top z^*-f^*(-\bar b)$: thus, the latter corresponds to the infimum of \eqref{eq:zbtinY_G} and the relaxation is tight.  
\end{proof}

\subsection{The general case}\label{sec:disconnected}
In this section, we give ideal formulations for $Z_Q$ when graph $G_Q$ in Definition~\ref{def:G} has several connected components. 
Given the graph $G_Q = (V, E)$, let $V_1, V_2, \dots, V_k$ be the vertex partition of connected components of $G_Q$. Let $\beta_{V_\ell}$ represent the subvector of $\beta$ corresponding to indices $V_\ell$.  Then
\[
\forall (z, \beta, t) \in Z_{Q}, f(\mathbf 1^{\top} \beta)=\sum_{\ell=1}^k f(\mathbf 1^{\top} \beta_{V_\ell}),
\]
because we cannot have two indices $i,j$ from different connected components such that $z_i = z_j = 1$. In other words, if $\beta_i \neq 0$ for some $i \in V_\ell, \ell\in[k]$, then $\beta_j = 0$ for all $j \in [p] \backslash V_\ell$.

For any $\ell=1,\ldots,k$, define the projection of the binary set $Q$ onto $V_\ell$ as 
\[
Q_\ell = \{ z \in \{0,1\}^{p} \; | \; z \in Q, z_i = 0, \; \forall i \notin V_\ell\},
\]
let $Q_\ell^0=Q_\ell \setminus \{\mathbf 0\}$ and note that, using arguments identical to those of Proposition~\ref{prop:F}, each $\conv(Q_\ell^0)$ admits a description as 
$$\conv(Q_\ell^0)=\conv(Q_\ell) \bigcap \{ z \; | \; \pi^\top z \geq 1, \; \forall\pi \in  \mathcal F_\ell\}$$ for some finite sets $\mathcal F_\ell \subseteq \R^p$. Note that $k>1$ and $\mathbf 0\not\in Q_\ell$ for some $\ell\in [k]$ implies that for all $z\in Q$, $z_i=0$ whenever $i\not\in V_\ell$. Therefore, we assume that $\mathbf 0\in Q_\ell$ and $\mathcal F_\ell\neq \emptyset$ for all $\ell\in [k]$. Furthermore, note that $\conv(Q_\ell)$ can be described as a system of linear inequalities, i.e., $A^\ell z^\ell \leq \delta^\ell$ for all $\ell\in [k]$.

We now give the main result of this section, namely a tight extended formulation for $\clconv(Z_{Q})$ when $G_Q$ has several connected components.
	\begin{theorem}\label{thm:notconnectedrevised}
		\small
\revised{		\begin{align*}
		\clconv(Z_{Q}) = \proj_{(z,\beta,t)}&\Big\{(z,\beta,\hat z,\hat \beta,\alpha,\hat t,t) \in [0,1]^p \times \R^{p}\times  [0,1]^{pk} \times \R^{pk}\times \R_+^k\times \R^k\times \R \; | \; \\
	&	 \sum_{\ell=1}^{k}\alpha_\ell=1, t=\sum_{\ell=1}^k \hat t^\ell, z=\sum_{\ell=1}^k \hat z^\ell, \beta=\sum_{\ell=1}^k \hat \beta^\ell,  A^\ell \hat z^\ell \leq \delta^\ell \alpha_\ell,\;\forall \ell\in [k],\\
	&	\hat t^\ell  \geq \alpha_\ell f\left( \frac{\mathbf 1^\top \hat \beta^\ell}{\alpha_\ell}  \right),\; \hat  t^\ell  \geq (\pi^\top \hat z^\ell)  f\left( \frac{\mathbf 1^\top \hat \beta^\ell}{\pi^\top \hat z^\ell} \right),\ \forall \pi \in \mathcal F_\ell,\;\forall \ell\in [k] \Big\}.
		\end{align*}
}
		\normalsize
	\end{theorem}

\begin{proof}
Observe that $Z_{Q} = \bigcup_{\ell = 1}^{k} Z_{Q_\ell}$ and by Theorem \ref{theo:hullConnected}, $(\hat t^\ell, \hat \beta^\ell, \hat z^\ell) \in \clconv (Z_{Q_\ell})$ if and only if
\begin{align*}
& f(\mathbf 1^\top \hat \beta^\ell) -\hat t^\ell \leq 0,\\ 
& (\pi^\top \hat z^\ell)f\left(\frac{\mathbf 1^\top\hat \beta^\ell}{\pi^\top\hat z^\ell}\right) -\hat t^\ell \leq 0,  \forall \pi \in \mathcal F_\ell\\
& \hat z^\ell\in \conv(Q_\ell).
\end{align*}
Now we see that $\clconv (Z_{Q_\ell})$ has a representation in the form $$\clconv (Z_{Q_\ell})= \{(\hat t^\ell, \hat\beta^\ell, \hat z^\ell) \; | \; G^\ell (\hat t^\ell, \hat\beta^\ell, \hat z^\ell) \leq 0\},$$ where each component function of $G^\ell$ is closed and convex. Then using Theorem 1 in \cite{ceria1999convex}, we obtain a description of $\clconv (Z_{Q})$ in a higher-dimensional space by taking the perspective of $G^\ell$:
\begin{subequations}
\begin{align}
z&=\sum_{\ell=1}^k \hat z^\ell \label{eq:disj_z}\\
\beta&=\sum_{\ell=1}^k \hat\beta^\ell \label{eq:disj_b}\\
t&=\sum_{\ell=1}^k \hat t^\ell\\
1&=\sum_{\ell=1}^k\alpha_\ell\\
\hat t^\ell & \geq \alpha_\ell f\left( \frac{\mathbf 1^\top\hat \beta^\ell}{\alpha_\ell}  \right) &\forall \ell\in [k]\\
\hat t^\ell & \geq (\pi^\top \hat z^\ell)  f\left( \frac{\mathbf 1^\top \hat\beta^\ell}{\pi^\top\hat z^\ell} \right)  &\forall \ell \in [k], \pi \in \mathcal F_\ell\\
&A^\ell \hat z^\ell \leq \delta^\ell \alpha_\ell &\forall \ell \in [k].
\end{align}	
\end{subequations}
 \revised{Hence, the result follows. }
 
\end{proof}

\section{Special Cases}\label{sec:cases}

In this section, we use Theorems~\ref{theo:hullConnected} and \ref{thm:notconnectedrevised} to derive ideal formulations for $Z_Q$ under various constraints defining $Q$.  Direct proofs of Propositions \ref{theo:cardinality},~\ref{theo:k1} and~\ref{theo:hier} were given in the preliminary version of this paper \cite{WGK20} for the special case of convex quadratic functions.

\subsection{Unconstrained case}
Consider the unconstrained case where $Q_{\text u}= \{0,1\}^p$ and $$Z_{Q_{\text u}}=\left\{(z,\beta,t)\in\{0,1\}^p\times \R^{p+1}~|~ f(h^\top\beta)\leq t, \beta_i(1-z_i)=0,\; \forall i\in [p]\right\}.$$ 
\begin{proposition}\label{prop:unconstrained}
	$$\clconv(Z_{Q_{\text u}})=\left\{(z,\beta,t)\in[0,1]^p\times \R^{p+1}~|~f(h^\top\beta)\leq t,\; (\mathbf 1^\top z)f\left(\frac{h^\top \beta}{\mathbf 1^\top z}\right)\leq t\right\}.$$
\end{proposition}
\begin{proof}
In this case set $Q_{\text u}^0=\{0,1\}^p\setminus\{\mathbf 0\}$ and $\conv(Q_{\text u}^0)=\{z\in [0,1]^p~|~ \mathbf 1^\top z\geq 1\}$. Thus $\mathcal{F}=\{\mathbf 1\}$ in Theorem~\ref{theo:hullConnected}, corresponding to the valid inequality $\mathbf 1^\top z\geq 1$ defining $\conv(Q_{\text u}^0)$, and the result follows.   
\end{proof} 

Note that Proposition~\ref{prop:unconstrained} generalizes existing results in the literature: if $p=1$ and function $f$ is one-dimensional, then Proposition~\ref{prop:unconstrained} reduces to the perspective reformulation \citep{ceria1999convex}; if $p\geq 2$ and $f$ is quadratic, then Proposition~\ref{prop:unconstrained} reduces to the rank-one strengthening derived in \cite{atamturk2019rank}.

\subsection{Cardinality constraint}\label{sec:cardinality}
Consider sets defined by the cardinality constraint, 
$$Q_{\text c}=\left\{z\in \{0,1\}^p~|~\mathbf 1^\top z\leq q \right\}.$$
Clearly, $\conv(Q_{\text c})=\left\{z\in [0,1]^p~|~\mathbf 1^\top z\leq q \right\}$ for any positive integer $q$. 
We now prove that, under mild conditions, ideal formulations are achieved by strengthening only the nonlinear objective.

\begin{proposition}\label{theo:cardinality}  If $q\geq 2$ and integer, then 
 \begin{align*}\clconv(Z_{Q_{\text c}}) = \Big\{(z,\beta,t) \in [0,1]^{p} &\times \mathbb R^{p+1}~|~\mathbf 1^\top z \leq q,\; f(h^\top \beta) \leq t,\\
 &\left(\mathbf 1^\top z\right)f\left(\frac{h^\top \beta}{\mathbf 1^\top z}\right) \leq t \Big\}.\end{align*}
\end{proposition}
\begin{proof}
Note that if $q\ge 2$, then $G_{Q_{\text c}}$ is a complete graph, hence $i\sim j$ for all $i,j\in[p], i\ne j$. Furthermore, $\conv(Q_{\text c}^0)=\{z \in [0,1]^{p}: 1\le \mathbf 1^\top z \leq q\}$. Hence $\mathcal F=\{\mathbf 1\}$. Then the result follows from Theorem~\ref{theo:hullConnected}.  
\end{proof}

The assumption that $q\geq 2$ in Proposition~\ref{theo:cardinality} is necessary. As we show next, if $q=1$, then it is possible to strengthen the formulation with a valid inequality  that uses the information from the cardinality constraint, which was not possible for $q>1$. Note that the case $q=1$ is also of practical interest, as set $Q_{\text c}$ with $q=1$ arises for example when preventing multi-collinearity \cite{bertsimas2016} or when handling nested categorical variables \cite{carrizosa2020linear}.  
\begin{proposition} \label{theo:k1}If $q=1$, then 
 \begin{equation*}
 \clconv(Z_{Q_{\text c}}) = \left\{(z,\beta,t) \in [0,1]^{p} \times \mathbb R^{p+1} ~  | ~ \mathbf 1^\top z \leq \revised{q},\; \sum_{i \in [p]} z_if\left(\frac{h_i\beta_i}{z_i}\right) \leq t \right\}.
 \end{equation*}
\end{proposition}
\begin{proof}
	First, observe that if $q=1$, then $G_{Q_{\text c}}$ is fully disconnected and it  decomposes into $p$ nodes, one for each variable $z_i, i\in [p]$: thus, in Theorem~\ref{thm:notconnectedrevised}, we find that $\hat z_i^\ell\neq 0$ and $\hat \beta_i^\ell\neq 0$ if and only if $\ell=i$. In addition, because each component  $\ell\in[p]$ has a single variable $\hat z_i^{i}$ for $\ell=i$,  $A^\ell\hat z^\ell\leq \delta^\ell $ is given by  $\hat z_i^{i}\leq 1$. Moreover, we find that $\mathcal{F}_i=\{1\}$ for all $i\in [p]$ in Theorem~\ref{thm:notconnectedrevised}. Thus, from Theorem~\ref{thm:notconnectedrevised}, we find that

			\begin{align*}
		\clconv(Z_{Q_{\text c}}) = \proj_{(z,\beta,t)}&\Big\{(z,\beta,\hat z,\hat \beta,\alpha,\hat t,t) \;| \; 
		\sum_{i=1}^{n}\alpha_i=1, t=\sum_{i=1}^n \hat t^i,\\ &z_i=\hat z_i^i, \; \beta_i= \hat \beta_i^i,\; \hat z_i^i \leq  \alpha_i,\;\forall i\in [p],\\
		&	\hat t^i  \geq \alpha_i f\left( \frac{\hat h_i\beta_i^i}{\alpha_i}  \right),\; \hat  t^i  \geq \hat z_i^i  f\left( \frac{h_i\hat \beta_i^i}{\hat z_i^i} \right),\;\forall i\in [p] \Big\}.
		\end{align*}

	Constraints $\hat z_i^i\leq \alpha_i$ imply that $\hat z_i^i  f\left( \frac{\hat \beta_i^i}{\hat z_i^i} \right)\geq \alpha_i f\left( \frac{\hat \beta_i^i}{\alpha_i}  \right)$. Finally, variables $\hat z_i^i$ and $ \hat \beta_i^i$ can be substituted with $z_i$ and $\beta_i$, variables $\alpha_i$ can be projected out (resulting in the inequality $\mathbf 1^\top z\leq 1$), and the result follows.  
\end{proof}

\subsection{Strong hierarchy constraints}
We now consider the hierarchy constraints. Hierarchy constraints arise from regression problems under the model \eqref{eq:constrainedRegression}, where the random variables include individual features as well as variables representing the interaction (usually pairwise) between a subset of these features given by a collection $\mathcal P$ of subsets of $[p]$. More formally, let the random variable $\theta(S)$ represent the (multiplicative) interaction of the features $i\in S$ for some subset $S\subseteq [p]$. 
Under this setting, the strong hierarchy constraints 
\begin{equation}\label{eq:strhier}
\theta(S)\neq 0\implies \beta_i\neq 0, \; \forall i\in S
\end{equation}
have been shown to improve statistical performance \cite{bien2013lasso,hazimeh2019learning} by ensuring that interaction terms are considered only if all corresponding features are present in the regression model. Strong hierarchy constraints can be enforced via the constraints $z(S)\leq z_i$ for all $i\in S$, where $z(S)\in \{0,1\}$ is an  indicator variable such that $\theta(S)(1-z(S))=0$. Thus, in order to devise strong convex relaxations of problems with hierarchy constraints, we study the set
$$Q_{\textrm{sh}}=\left\{z \in \{0,1\}^{p} ~ | ~ z_p \leq z_i, \; \forall i \in [p-1] \right\}.$$ 
Note that in $Q_{\textrm{sh}}$ we identify $S$ with $[p-1]$, $z(S)$ with $z_p$ and $\theta(S)$ with $\beta_p$; since $p$ is arbitrary, this identification is without loss of generality.

To establish the convex hull of  $Z_{Q_{\textrm{sh}}}$,  we give a lemma that characterizes $\conv(Q_{\textrm{sh}}^0).$ First, observe that 
   \begin{equation}\label{eq:strongF}\sum_{i \in [p-1]} z_i  - (p-2) z_p \geq 1\end{equation} 
is a  valid inequality for $Q_{\textrm{sh}}^0$. To see this, note that for $z\ne \mathbf 0$, if $z_p=0$, then we must have $\sum_{i \in [p-1]} z_i \ge 1$, and if $z_p=1$, then we must have $\sum_{i \in [p-1]} z_i =p- 1 $, so the validity follows.

\begin{lemma}\label{lem:integralQ}
$$\Conv(Q_{\textrm{sh}}^0) = \left\{z\in [0,1]^{p} ~|~ \sum_{i \in [p-1]} z_i  - (p-2) z_p \geq 1,\; z_p \leq z_i,\ \forall  i \in [p-1] \right\}.$$
\end{lemma}

\begin{proof}
Let  $$Q_g = \left\{z\in [0,1]^{p} ~|~ \sum_{i \in [p-1]} z_i  - (p-2) z_p \geq 1,\; z_p \leq z_i,\ \forall  i \in [p-1] \right\}.$$  We will first show that the extreme points of $Q_g$ are integral. Then we will prove that $\conv(Q_{\textrm{sh}}\setminus \{\mathbf 0\})=Q_g.$

Suppose $z^{*}$ is an extreme point of $Q_g$. 
Observe that if $z_p^{*}$ is equal to $1$, then $z_i^{*}=1$ for all $i\in [p-1]$. If $z_p^{*}$ is equal to $0$, then the constraint matrix defining $Q_g$ is totally unimodular, thus all extreme points of $Q_g$ with $z_p^{*}=0$ are integral. If constraint \eqref{eq:strongF} is not tight at an extreme point, then because the remaining constraint matrix defining  $Q_g$ is totally unimodular,  the corresponding extreme point of $Q_g$  is integral. Therefore, it suffices to consider extreme points where \eqref{eq:strongF} holds at equality and $0<z_p^{*}<1$.

Now suppose $\sum_{i \in [p-1]} z^{*}_i  - (p-2)z^{*}_p = 1$ and 
 $1 > z^{*}_p > 0$.   We first show that $z_i^{*}=1$ for at most one coordinate $i\in [p-1]$. If $z^{*}_i =z_j^{*}=1 $ for $i\ne j$, then 
 \begin{equation}\sum_{\ell \in [p-1]} z^{*}_\ell  - (p-2)z^{*}_p = z^{*}_i + \sum_{\ell \in [p-1], \ell \neq i} (z^{*}_{\ell} - z^{*}_p) \geq z^{*}_i + (z^{*}_{j} - z^{*}_p) > z^{*}_i = 1,\label{eq:two1}\end{equation} where the first inequality follows from dropping terms $z_\ell^*-z_p^*\geq 0$ with $\ell\neq j$, and the second inequality follows from the assumption $z_j^*=1$ and $z_p^*<1$. Since \eqref{eq:two1} contradicts $\sum_{i \in [p-1]} z^{*}_i  - (p-2)z^{*}_p = 1$, it follows that $z_i^{*}=1$ for at most one coordinate $i\in [p-1]$. 
 
 Next, observe that  if $z^{*}_i = z^{*}_p$  for all  $i \in [p-1]$, then $\sum_{i \in [p-1]} z^{*}_i  - (p-2)z^{*}_p = z^{*}_p < 1$. Therefore, the largest element in $z^{*}_i, i \in [p-1]$ has to be strictly greater than $z^{*}_p$. 
Finally, we now show that  we can perturb $z^{*}_p$ and the $p-2$ smallest elements in $z^{*}_i, i \in [p-1]$ by a small quantity $\epsilon$ and remain in $Q_g$.   The equality $\sum_{i \in [p-1]} z_i  - (p-2) z_p = 1$ clearly holds after the perturbation. And, adding a small quantity $\epsilon$ to $z^{*}_p$ and the $p-2$ smallest elements in $z^{*}_i, i \in [p-1]$ does not violate the hierarchy constraint since the largest element in $z^{*}_i, i \in [p-1]$ is strictly greater than $z^{*}_p$. 
   Finally, since $z^{*}_i \ge z^{*}_p >0, \forall i \in [p-1]$, subtracting a small quantity $\epsilon$ does not violate the non-negativity constraint. Thus, we can write $z^{*}$ as a convex combination of two points in $Q_g$, which is a contradiction. 
  
  To see  that $Q_g=\conv(Q_{\textrm{sh}}^0)$, first, observe that $\mathbf 0\not \in Q_g$. Also, 
 \eqref{eq:strongF} is a valid inequality for $Q_{\textrm{sh}}^0$.  Furthermore, we just showed that the extreme points of $Q_g$ are integral, hence  
$Q_g=\conv(Q_{\textrm{sh}}^0)$.

  \end{proof}

Now we are ready to give an ideal formulation for $Z_{Q_{\textrm{sh}}}$.

\begin{proposition}\label{theo:hier} The closure of the convex hull of $Z_{Q_{\textrm{sh}}}$ is given by
	\begin{equation*}
	\begin{split}
	\clconv(Z_{Q_{\textrm{sh}}}) = \Big\{&(z,\beta,t) \in [0,1]^{p} \times \mathbb R^{p+1}~|~  f(h^\top \beta) \leq t,\  z_p \leq z_i, \forall i \in [p-1], \\
	& \left(\sum_{i \in [p-1]} z_i - (p-2) z_p\right)f\left(\frac{h^\top \beta}{\sum_{i \in [p-1]} z_i - (p-2) z_p}\right) \leq t \Big\}.
	\end{split}
	\end{equation*}
\end{proposition}

\begin{proof}
First, observe that the constraint matrix defining $Q_{\textrm{sh}}$ is totally unimodular, so $\conv(Q_{\textrm{sh}})=\left\{z \in [0,1]^{p} ~ | ~ z_p \leq z_i, \; \forall i \in [p-1] \right\}.$ Note that $G_{Q_{\textrm{sh}}}$ is a complete graph, hence $i\sim j$ for all $i,j\in[p], i\ne j$.  
Hence, from Lemma~\ref{lem:integralQ},  $\mathcal F=\{(1,\dots,1, -(p-2)) \}$. Then the result follows from Theorem~\ref{theo:hullConnected}.
  \end{proof}

\subsection{Weak hierarchy}
Consider the strong hierarchy relation \eqref{eq:strhier}, which requires all variables in the set $S$ to have non-zero coefficients to capture a multiplicative effect, $\theta(S)$ on the response variable $y$. The weak hierarchy relation \cite[]{bien2013lasso}   is a relaxation of the strong hierarchy relation to address the interaction between random variables in the same subset $S$ by requiring
$$\theta(S)\neq 0\implies \beta_i\neq 0, \; \text{ for \emph{some} } i\in S.$$
Using similar arguments as before, we formulate the weak hierarchy relation  as 
$
z_p \leq \sum_{i \in [p-1]} z_i,
$
in other words,
$z_1, z_2, \dots, z_{p-1} = 0 \implies z_p = 0.$ 
The corresponding constrained indicator variable set is thus defined by
\[
Q_{\textrm{wh}} = \left\{ z \in \{0,1\}^{p} \; | \; z_p \leq \sum_{i \in  [p-1]} z_i\right\}.
\]
Note that $\mathbf 1 \in Q_{\textrm{wh}}$, thus the graph $G_{Q_{\textrm{wh}}}$ is connected and Theorem~\ref{theo:hullConnected} can be used to derive the convex hull.

\begin{proposition}\label{theo:weakh}
	\begin{align*}\clconv(Z_{Q_{\textrm{wh}}}) = \Big\{(z,\beta,t) \in [0,1]^{p}& \times \mathbb R^{p+1} ~ | ~ f(h^\top \beta) \leq t, z_p \leq \sum_{i \in  [p-1]} z_i,\\
	 &\left(\sum_{i \in [p-1]} z_i\right)f\left(\frac{h^\top \beta}{\sum_{i \in [p-1]} z_i}\right) \leq t \Big\}.\end{align*}
\end{proposition}
\begin{proof}
First, observe that  the constraint matrix defining $Q_{\textrm{wh}}$ is totally unimodular, hence $\conv(Q_{\textrm{wh}}) = \left\{ z \in [0,1]^{p} \; | \; z_p \leq \sum_{i \in  [p-1]} z_i\right\}.$
Clearly, $\sum_{i \in [p-1]} z_i \geq 1$
	is valid for $Q_{\textrm{wh}}^{0}$ since $z_1 = \cdots =z_{p-1}=0 \implies z_p = 0$. 
	It suffices to show that 
	\begin{equation}\label{eq:constraintsWeak}
	\conv(Q_{\textrm{wh}}^{0}) = \left\{ z \in [0,1]^{p} | \sum_{i \in [p-1]} z_i \geq 1 \right\}.
	\end{equation}
	 All extreme points of the polyhedron on the right-hand side of \eqref{eq:constraintsWeak}
	 are integral, because the associated constraint matrix is an interval matrix with integral right-hand side. The result follows from Theorem \ref{theo:hullConnected}. 
	 
\end{proof}

\section{A note on separable functions}\label{sec:separable}

In this section, we demonstrate that the proof technique used in \S\ref{sec:convexification} can be extended to separable functions with constraints, resulting in relatively simple proofs generalizing existing results in the literature. 

Given a partition of $[p]=\bigcup_{j=1}^\ell V_j$ and convex functions $f_j:\R^{V_j}\to \R$ such that $f_j(\mathbf 0)=0$, consider the  epigraph of a separable function of the form:
	\begin{align*}
	W = \Big\{ z\in Q\subseteq \{0,1\}^\ell, \beta\in \R^p, t\in \R ~|~&  \sum_{j=1}^\ell f_j(\beta_{V_j}) \leq t ,\\
	 &\beta_i (1 - z_j) = 0,\; \forall j\in [\ell], i\in V_j\Big\}. 
	\end{align*}
	As Theorem~\ref{theo:separable} below shows, ideal formulations of $W$ can be obtained by applying the perspective reformulation on the separable nonlinear terms and, \emph{independently}, strengthening the continuous relaxation of $Q$. 
	Let 
	$$Y_s=\left\{ (z,\beta,t) \in \mathbb R^{\ell+p+1} ~ | ~ \sum_{j=1}^\ell z_jf_j\left(\frac{\beta_{V_j}}{z_j}\right) \leq t,\; z \in \conv(Q) \right\}.$$
	
	\begin{theorem}\label{theo:separable} $Y_s$ is the closure of the convex hull of $W$: 
		$\clconv(W)=Y_s.$
	\end{theorem}
	
	\begin{proof}
		Validity of the corresponding  inequality in $Y_s$ follows directly from the validity of the perspective reformulation. 
		For any $(a,b,c) \in \mathbb R^{\ell+p+1}$ consider the following two problems
		\begin{align}\label{eq:zbtinX}
		\min \quad &a^\top  z + b^\top  \beta + c t  \qquad~\mbox{subject to\;\;\;\;} (z, \beta, t) \in W,
		\end{align}
		and
		\begin{align}\label{eq:zbtinY}
		\min \quad &a^\top  z + b^\top  \beta + c t  \qquad~\mbox{subject to\;\;\;\;} (z, \beta, t) \in Y_s.
		\end{align}
		It suffices to show that \eqref{eq:zbtinX} and \eqref{eq:zbtinY} are equivalent, i.e., there exists an optimal solution of \eqref{eq:zbtinY} that is optimal for \eqref{eq:zbtinX} with the same objective value. As before, we may assume that $c = 1$ without loss of generality. For $j\in [\ell]$, let $f_j^*:\R^{V_j}\to \R$ be the convex conjugate of function $f_j$, i.e.,  $$f_j^*(\gamma)=\sup_{\beta\in \R^{V_j}}\gamma^\top\beta-f_j(\beta),$$ and let $\Gamma_j=\left\{\gamma\in \R^{V_j}: f_j^*(\gamma)<\infty \right\}.$  
		From Fenchel's inequality corresponding to the perspective function, we find that for any $\beta\in \R^{V_j}$, $z_j\geq 0$ and $\gamma\in \Gamma_j$, 
		\begin{equation}\label{eq:fenchel2}
		z_jf_j\left(\frac{\beta}{z_j}\right)\geq \gamma^\top\beta-z_jf_j^*(\gamma). 
		\end{equation}
		
		Observe that both \eqref{eq:zbtinX} and \eqref{eq:zbtinY} are unbounded if $-b_{V_j}\not \in \Gamma_j$ for some $j\in [\ell]$. Otherwise, if $-b_{V_j}\in \Gamma_j$ for all $j\in [\ell]$, we use \eqref{eq:fenchel2} with $\gamma=-b_{V_j}$ for each $j\in [\ell]$ to lower bound the objective of \eqref{eq:zbtinY}, resulting in the relaxation
		\begin{subequations}\label{eq:zbtinconvQ}
			\begin{align}
			\min \quad & \sum_{j=1}^\ell \Big(a_j-f_j^*(-b_{V_j})\Big)z_j \\
			\text{s.t.} \quad& z \in \conv(Q),
			\end{align}
		\end{subequations}
	which admits an optimal solution $z^*\in Q$. Letting $\beta_{V_j}^*\in {\arg\sup}_{\beta\in \R^{V_j}}-b_{V_j}^\top\beta-f_j(\beta_{V_j})$ whenever $z_j^*=1$ and $\beta_{V_j}^*=\mathbf 0$ otherwise, we find a feasible solution for \eqref{eq:zbtinX} with the same objective value.
		  \end{proof}
	
		Theorem~\ref{theo:separable}  generalizes the result of \citet{xie2018ccp} for  \linebreak $Q=\left\{z\in \{0,1\}^p ~|~ \sum_{i=1}^p z_i\leq q\right\}$,  $V_j=\{j\}$, and $f_j(\beta_j)=\beta_j^2$ for $j\in[p]$.  
Theorem~\ref{theo:separable}  also generalizes the result of	\citet{baccinew2019} for the case that $f_j$ is convex, differentiable and certain constraint qualification conditions hold,  \revised{applied to our setting}.  \revised{However, \citet{baccinew2019} consider more general settings where multiple polyhedra are connected by a single binary variable, and under linear constraints on the continuous variables.} 
	
\section{Quadratic Case: Implementation via Semidefinite Optimization}	\label{sec:quadratic}
In this section we review how to implement the convexifications derived in \S\ref{sec:convexification} for the special case of quadratic optimization. \revised{Given observations $(x_i,y_i)_{i=1}^n$ with $x_i\in \R^p$ and $y_i\in \R$, let $X_{n\times p}$ defined as $X_{ij}=(x_i)_j$ be the model matrix, and} consider least square regression problems
\begin{subequations}\label{eq:leastsquares}
\begin{align}
\min_{z,\beta}\;&\|y-X\beta\|_2^2+\lambda\|\beta\|_2^2+\mu\|z\|_1\\
\text{s.t.}\;&\beta_i(1-z_i)=0&\forall i\in [p]\label{eq:compl}\\
&\beta\in \R^p,\; z\in Q\subseteq\{0,1\}^p,\label{eq:leastsquares_bounds}
\end{align}
\end{subequations}
where the regularization terms $\lambda\|\beta\|_2^2$ and  $\mu\|z\|_1$ penalize the $\ell_2$-norm and  $\ell_0$-norm of $\beta$, respectively. \revised{A natural convexification of \eqref{eq:leastsquares} based on the $\ell_2$-regularization term $\lambda\|\beta\|_2^2$ is to directly use the perspective relaxation \cite{bertsimas2017sparse,xie2018ccp}
\begin{subequations}\label{eq:leastsquaresPersp}
	\begin{align}
	\min_{z,\beta}\;&\|y-X\beta\|_2^2+\lambda\sum_{i=1}^p t_i+\mu\|z\|_1\\
	\text{s.t.}\;
	&\beta_i^2\leq t_i z_i&\forall i\in [p]\\
	&\beta\in \R^p,\; z\in \conv(Q).
	\end{align}
\end{subequations}
Formulation \eqref{eq:leastsquaresPersp} can either be directly implemented with conic quadratic solvers \cite{akturk2009strong}, implemented via cutting plane methods \cite{frangioni2009computational} or via tailored methods specific to linear regression \cite{hazimeh2020sparse}. However, \eqref{eq:leastsquaresPersp} is weak if $\lambda$ is small.
}

\revised{In this paper we focus on relaxations that do not assume the presence of the $\ell_2$-regularization term  (but require solving an SDP). In particular,}
letting $B\approx\beta\beta^\top$, \citet{dong2015regularization} propose the semidefinite relaxation of \eqref{eq:leastsquares} given by 
\begin{subequations}\label{eq:perspective}
	\begin{align}
	\min_{z,\beta,B}\;&\|y\|_2^2-2y^\top X\beta+\langle X^\top X+\lambda I, B\rangle+\mu\sum_{i=1}^pz_i \label{eq:perspective-obj}\\
	\text{s.t.}\;&\begin{pmatrix}z_i & \beta_i \\ \beta_i & B_{i,i}\end{pmatrix}\succeq 0&\forall i\in [p]\label{eq:perspective1}\\
	&\begin{pmatrix}1 & \beta^\top \\ \beta & B\end{pmatrix}\succeq 0\\
	&\beta\in \R^p,\; \revised{B\in \R^{p\times p}},\; z\in \conv(Q),\label{eq:perspective_bounds}
	\end{align}
\end{subequations}
which dominates the \revised{perspective relaxation \eqref{eq:leastsquaresPersp}, as well as any perspective relaxation obtained from extracting a diagonal matrix from $X^\top X+\lambda I$, e.g., using the method in \cite{frangioni2007sdp}}.  We now discuss how \eqref{eq:perspective} can be further strengthened.

Given any $T\subseteq [p]$, let $\beta_T$, $z_T$ and $B_T$ the subvectors of $\beta$ and $z$ and submatrix of $B$ induced by $T$, respectively. Moreover, let $Q_T$ be the projection of $Q$ onto the subspace of variables in $T$. First, observe that in order to apply our theoretical developments to this setting, we need to extract a convex function of the form $f(h^\top \beta_T)$ for some $h\in \R^{|T|}$. In particular, we consider quadratic $f$. Note that for any $h$,   from Theorem~\ref{theo:hullConnected}, we can obtain valid inequalities of the form 
\begin{equation}\label{eq:valid_quad}t \geq \frac{( h^\top  \beta_T)^2}{\pi^\top z_T},\; \; \forall \pi \in \mathcal F_T\end{equation}
for some set $\mathcal F_T\subseteq \R^{|T|}$ describing $Q_T^0$. Inequalities \eqref{eq:valid_quad} can  then be included in formulation \eqref{eq:perspective} by using the methodology given in \cite{han20202x2}, as discussed next.

For  any $h\in \R^{|T|}$, we find that for $z\in Q_T$ and $B_T=\beta_T\beta_T^\top$ satisfying \eqref{eq:compl},
\begin{equation}\label{eq:derivation}
\langle hh^\top,B_T\rangle= (h^\top \beta)^2\geq \frac{( h^\top  \beta_T)^2}{\pi^\top z_T}.
\end{equation}
Observe that inequality \eqref{eq:derivation} is valid for any vector $h$. Therefore, by optimizing over $h$ to find the strongest inequality, we obtain  
\begin{equation}\label{eq:sdp0}0\geq \max_{h\in \R^{|T|}}\left\{ \frac{( h^\top  \beta_T)^2}{\pi^\top z_T}-\langle hh^\top,B_T\rangle\right\}.
\end{equation}
\revised{Inequality \eqref{eq:sdp0} is satisfied if and only if $h^\top\left(\beta_T\beta_T^\top/\pi^\top z_T-B_T\right)h\leq 0$ for all $h\in \R^{|T|}$, or, equivalently, if $B_T-\beta_T\beta_T^\top/\pi^\top z_T\succeq 0$. Using Schur complement, we conclude that constraint \eqref{eq:sdp0} is equivalent to 
\begin{equation}\label{eq:sdp}
\begin{pmatrix}\pi^\top z_T & \beta_T^\top \\
\beta_T &B_T\end{pmatrix}\succeq 0.
\end{equation}
 }
Observe that inequalities \eqref{eq:perspective1} are in fact special cases of \eqref{eq:sdp} with $T=\{i\}, i\in[p]$.

\section{Numerical Results}

In this section, we provide numerical results to compare relaxations of regression problems. \revised{Specifically, in \S\ref{sec:leastSquares} we present computations with sparse least squares regression problems }with all pairwise (second-order) interactions and strong hierarchy constraints \cite{hazimeh2019learning}\revised{; in \S\ref{sec:logistic} we present computations with logistic regression.} The \revised{conic optimization} problems are solved with MOSEK 8.1 solver on a laptop with a 2.0 GHz  intel(R)Core(TM)i7-8550H CPU with 16 GB main memory.

\subsection{Least squares regression with hierarchy constraints}\label{sec:leastSquares}

\revised{In this section we focus on least squares regression problems with hierarchy constraints.} \rev{A usual approach to compute estimators to statistical inference problems is either to use the relaxation of a suitable convex relaxation directly, or to round the solution obtained from such convex relaxations, see for example \cite{atamturk2019rank,atamturk2018sparse,bertsimas2020sparse,bertsimas2020mixed,dong2015regularization,pilanci2015sparse,xie2018ccp}. Thus, as a proxy to evaluate the quality of the estimators obtained, we focus on the optimality gap provided by such approaches. In \S\ref{sec:computations_formulations} we discuss the relaxations used, and in \S\ref{sec:computations_rounding} we discuss a simple rounding heuristic, which guarantees that the produced solutions satisfy the hierarchy constraints.}

\subsubsection{Formulations}\label{sec:computations_formulations}

Given \revised{observations $(x_\ell,y_\ell)_{\ell=1}^n$}, we consider relaxations of the problem
\begin{subequations}\label{eq:hierarchy}
	\begin{align}
	\min_{z,\beta}\;&\sum_{\ell=1}^n\left(y_\ell-\sum_{i=1}^p x_{\ell i}\beta_i-\sum_{i=1}^p\sum_{j=i}^p x_{\ell i}x_{\ell j}\beta_{ij}\right)^2+\lambda\|\beta\|_2^2+\mu\|z\|_1 \\
	\text{s.t.}\;&\beta_i(1-z_i)=0\quad \quad \forall i\in [p]\\
	&\beta_{ij}(1-z_{ij})=0\quad \quad \forall i,j\in [p],\; i\leq j\\
	&z_{ii}\leq z_i\quad \quad \forall i\in [p]\label{eq:constUni1}\\
	&z_{ij}\leq z_i,\; z_{ij}\leq z_j\quad \quad \forall i,j\in [p],\; i\leq j\label{eq:constUni2}\\
	&\beta\in \R^{p(p+3)/2},\; z\in\{0,1\}^{p(p+3)/2}.
	\end{align}
\end{subequations}
We standardize the data so that all columns have $0$ mean and norm 1, i.e., $\|y\|_2=1 $, $\|X_i\|_2^2=1$ for all $i\in [p]$, and $\|X_i\circ X_j\|_2^2=1$ for all $i\leq j$ \revised{(where $X_i\in \R^n$ and $(X_i)_\ell=x_{\ell i}$)}. Note that constraints \eqref{eq:constUni1}-\eqref{eq:constUni2} are totally unimodular, hence $\conv(Q)$ in \eqref{eq:perspective_bounds} can be obtained simply by relaxing integrality constraints to $0\leq z\leq 1$. 

In addition to the \textbf{\revised{optimal} perspective} reformulation \eqref{eq:perspective}, we consider the following strengthenings.
\begin{description}
	\item[\textbf{Rank1}] Inequalities \eqref{eq:sdp} for all sets $T$ of cardinality 2 using the ``unconstrained" convexification given in Proposition~\ref{prop:unconstrained}. This formulation was originally proposed in \cite{atamturk2019rank}. The resulting semidefinite constraints are of the form \begin{align*}&\begin{pmatrix}z_i+z_j & \beta_i &\beta_j\\
	\beta_i & B_{i,i}& B_{i,j}\\
	\beta_j & B_{i,j}& B_{j,j}\end{pmatrix}\succeq 0,\; \begin{pmatrix}z_i+z_{jk} & \beta_i &\beta_{jk}\\
	\beta_i & B_{i,i}& B_{i,jk}\\
	\beta_{jk} & B_{i,jk}& B_{jk,jk}\end{pmatrix}\succeq 0,\\
	&\text{ or }\begin{pmatrix}z_{i_1i_2}+z_{j_1j_2} & \beta_{i_1i_2} &\beta_{j_1j_2}\\
	\beta_{i_1i_2} & B_{i_1i_2,i_1i_2}& B_{i_1i_2,j_1j_2}\\
	\beta_{j_1j_2} & B_{i_1i_2,j_1j_2}& B_{j_1j_2,j_1j_2}\end{pmatrix}\succeq 0.\end{align*}
	\item[\textbf{Hier}]  Inequalities \eqref{eq:sdp} for \revised{all sets $T$ linked by hierarchy constraints. Specifically, from constraints \eqref{eq:constUni1} we add constraints with $|T|=2$ of the form
	 $$\begin{pmatrix}z_i & \beta_{i} &\beta_{ii}\\
	\beta_i & B_{i,i}& B_{i,ii}\\
	\beta_{ii} & B_{i,ii}& B_{ii,ii}\end{pmatrix}\succeq 0.$$
	Moreover, from constraints \eqref{eq:constUni2}, linking the three variables $\beta_i$, $\beta_j$ and $\beta_{ij}$, we add constraints involving pairs of variables $\beta_i$ and $\beta_{ij}$ of the form 
	$$\begin{pmatrix}z_i & \beta_{i} &\beta_{ij}\\
	\beta_i & B_{i,i}& B_{i,ij}\\
	\beta_{ij} & B_{i,ij}& B_{ij,ij}\end{pmatrix}\succeq 0.$$
	Constraints involving pairs of variables $\beta_j$ and $\beta_{ij}$ are identical and added as well. Finally, constraints considering the three variables simultaneously are added, resulting in constraints with $|T|=3$ of the form 
	$$\begin{pmatrix}z_i+z_j-z_{ij} & \beta_{i} &\beta_{i}&\beta_{ij}\\
	\beta_i & B_{i,i}& B_{i,j}&B_{i,ij}\\
	\beta_{j} & B_{i,j}&B_{j,j}&B_{j,ij}\\
	\beta_{ij}& B_{i,ij}&B_{j,ij}&B_{ij,ij}\end{pmatrix}\succeq 0.$$
}	
	\item[\textbf{Rank1+hier}] All inequalities of both 
	\textbf{Rank1} and 
	\textbf{Hier}.
\end{description}

\subsubsection{Upper Bounds and Gaps}\label{sec:computations_rounding}

Given the solution of the convex relaxation, we use a simple rounding heuristic to recover a feasible solution to problem \eqref{eq:hierarchy}: we round $z_i$ and fix it to the nearest integer---observe that a rounded solution always satisfies hierarchy constraints \eqref{eq:constUni1}-\eqref{eq:constUni2}---, and solve the resulting convex optimization problem in terms of $\beta$. Given the objective value $\nu_\ell$ of the convex relaxation and $\nu_u$ of the heuristic, we can bound the optimality gap as $\text{gap}=\frac{\nu_u-\nu_\ell}{\nu_u}\times 100\%.$  

\subsubsection{Instances and parameters}

We test the formulations on \revised{ six datasets: Crime (from \cite{hastie2015statistical}), Diabetes (from \cite{efron2004least}), Housing, Wine\_quality (red), Forecasting\_orders and Bias\_correction (latter four from %the UCI Machine Learning Dataset 
\cite{Dua:2017})}. Table~\ref{tab:dataset} shows the number of observations $n$ and number of original regression variables $p$, as well as the total number of variables $p(p+3)/2$ after adding all second order interactions. Finally we use regularization values $(\lambda,\mu)=(0.01i, 0.01j)$ for all $0\leq i,j\leq 30$ with $i,j\in \mathbb{Z}_+$ \revised{for all datasets but Bias\_correction, for which we use for $1\leq i\leq 20$ and $1\leq j\leq 30$ with $i,j\in \mathbb{Z}_+$, due to its larger size and longer relaxation solution times}. 

\begin{table}[htb]
	\begin{center}
		\caption{Datasets.}
		\label{tab:dataset}
		\begin{tabular}{c| c c c}
			\hline
			\textbf{dataset}&${n}$&${p}$&${p(p+3)/2}$\\
			\hline
			Crime & 51 & 5 & 20 \\
			Diabetes & 442 & 10 & 65 \\
			Wine\_quality (red) & 1599 & 11 & 77\\
			Forecasting\_orders&60&12&90\\
			Housing & 507 & 13 & 104\\
			Bias\_correction & 7,590 & 18 & 189\\
			\hline
		\end{tabular}
	\end{center}
\end{table}

\subsubsection{Results}

\begin{figure}[!htb]
	\centering
	\subfloat[Crime]{\includegraphics[width=0.5\textwidth,trim={10.1cm 5.5cm 11cm 5cm},clip]{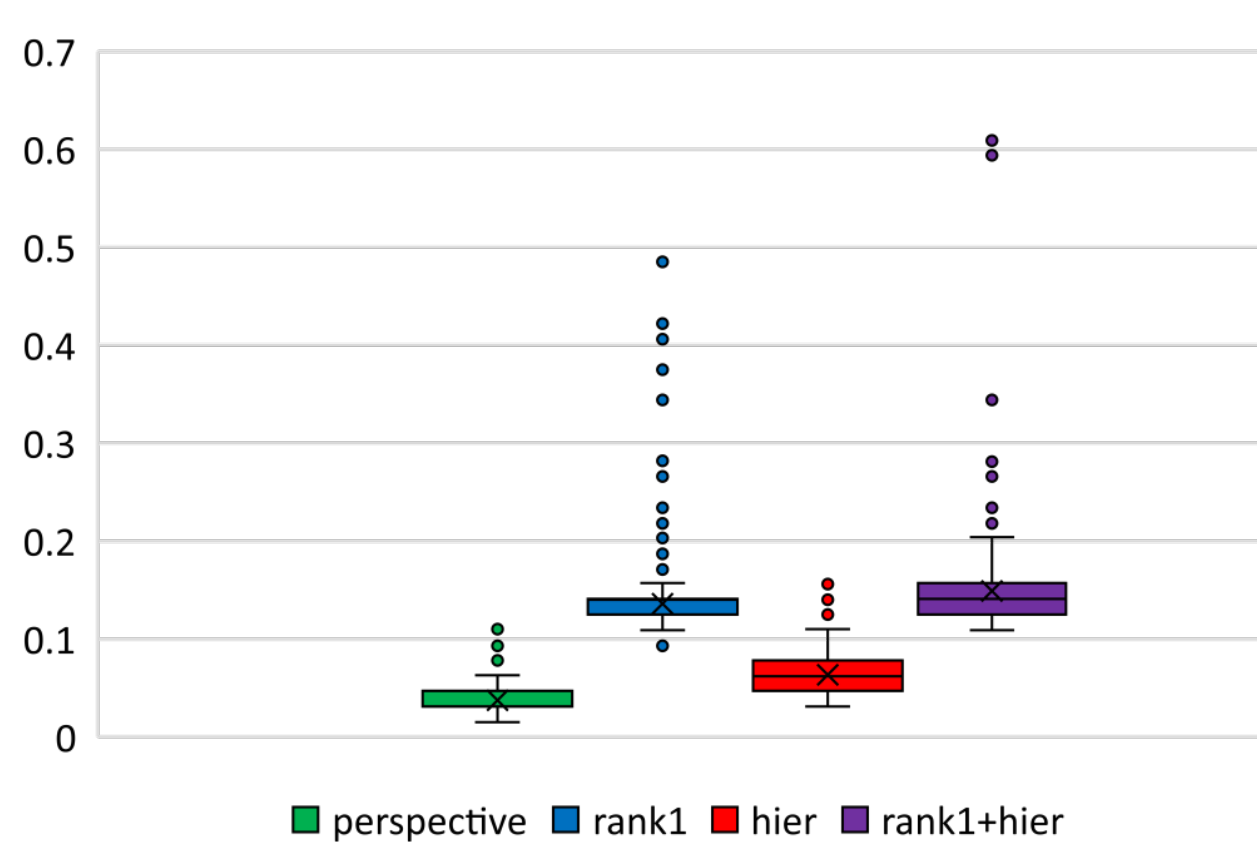}}\hfill\subfloat[Diabetes]{\includegraphics[width=0.5\textwidth,trim={10.1cm 5.5cm 11cm 5cm},clip]{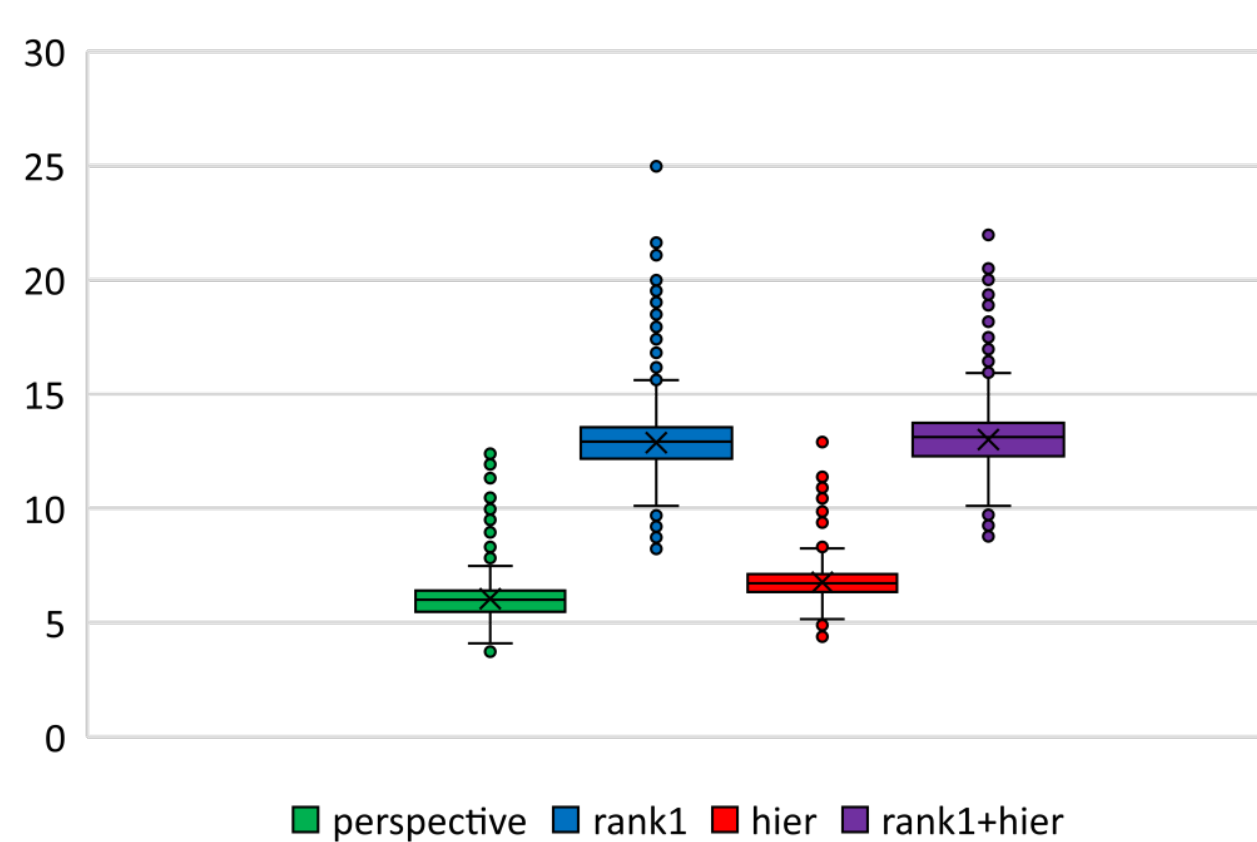}}\hfill
	\subfloat[Wine\_quality]{\includegraphics[width=0.5\textwidth,trim={10.1cm 5.5cm 11cm 5cm},clip]{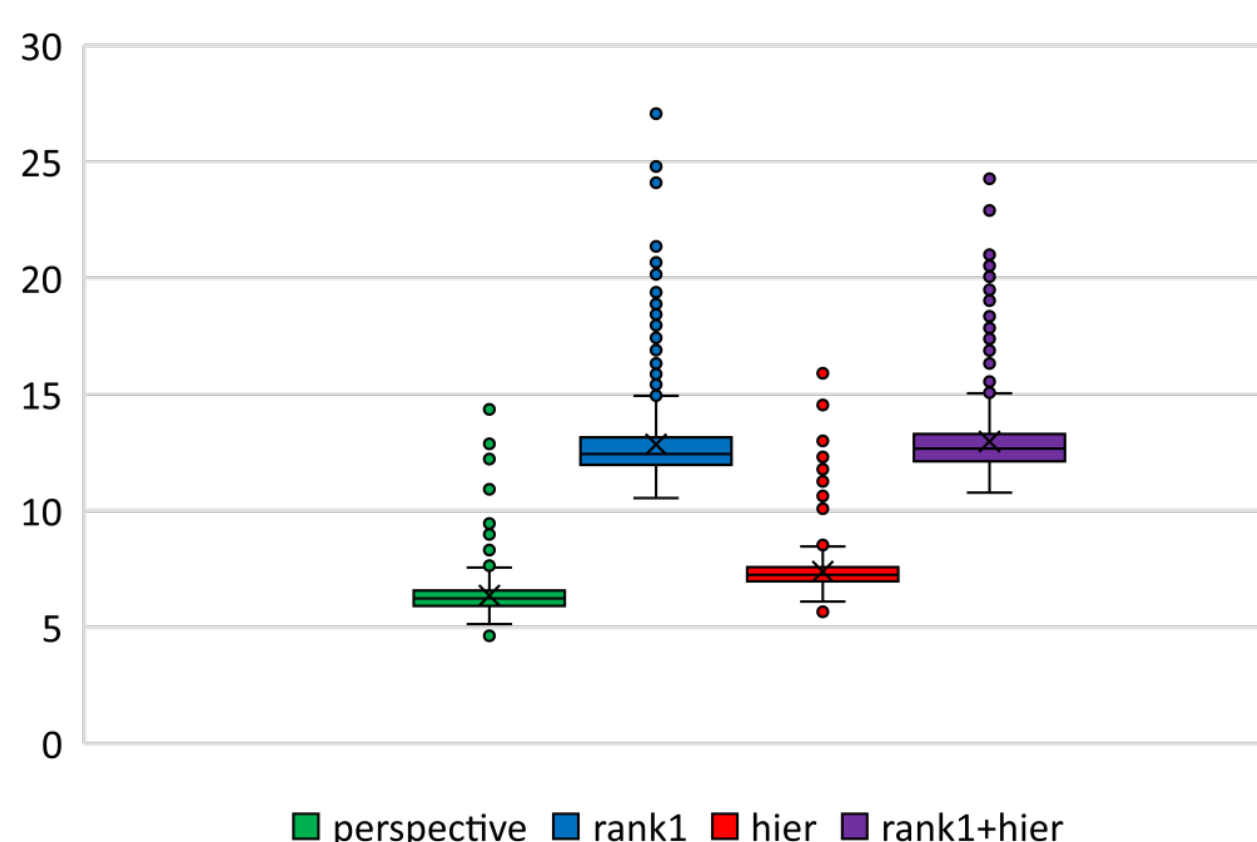}}\hfill\subfloat[Forecasting\_orders]{\includegraphics[width=0.5\textwidth,trim={10.1cm 5.5cm 11cm 5cm},clip]{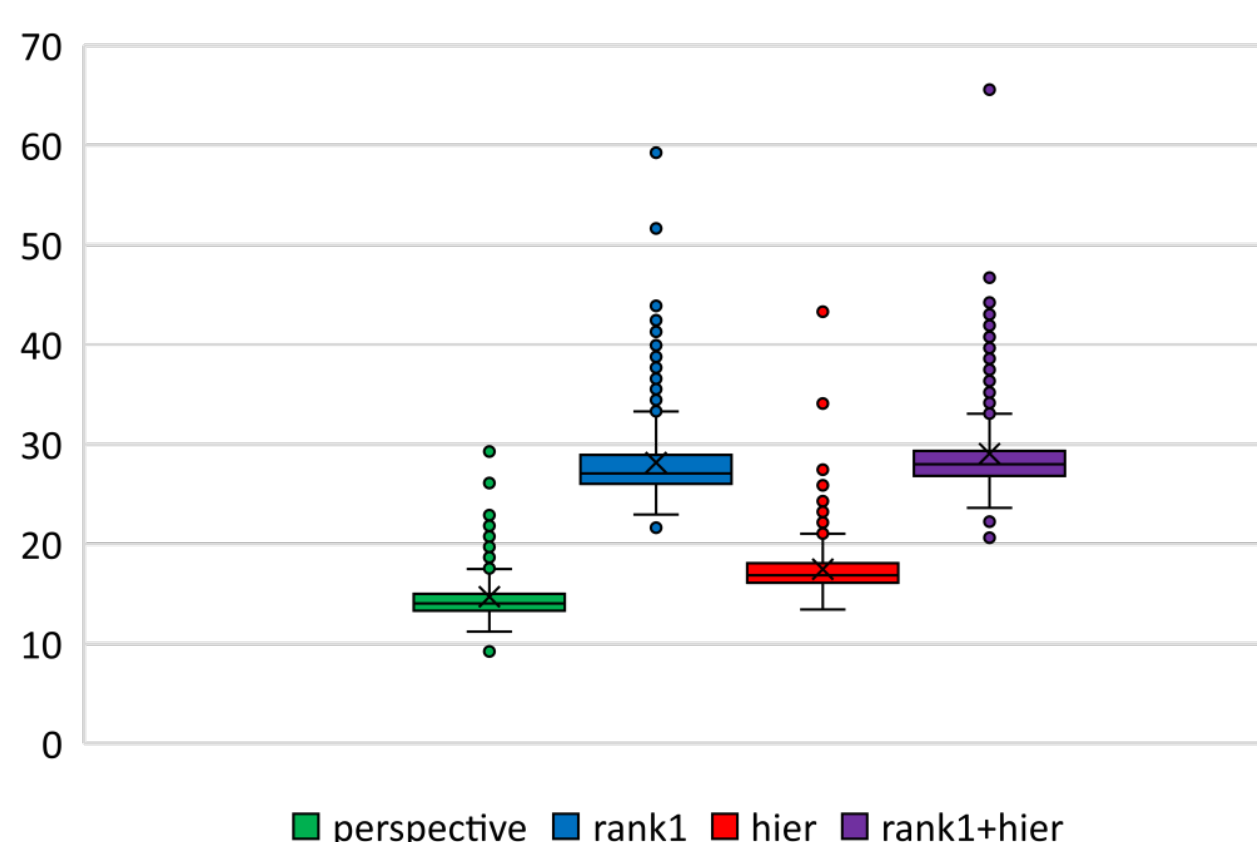}}\hfill\
	\subfloat[Housing]{\includegraphics[width=0.5\textwidth,trim={10.1cm 5.5cm 11cm 5cm},clip]{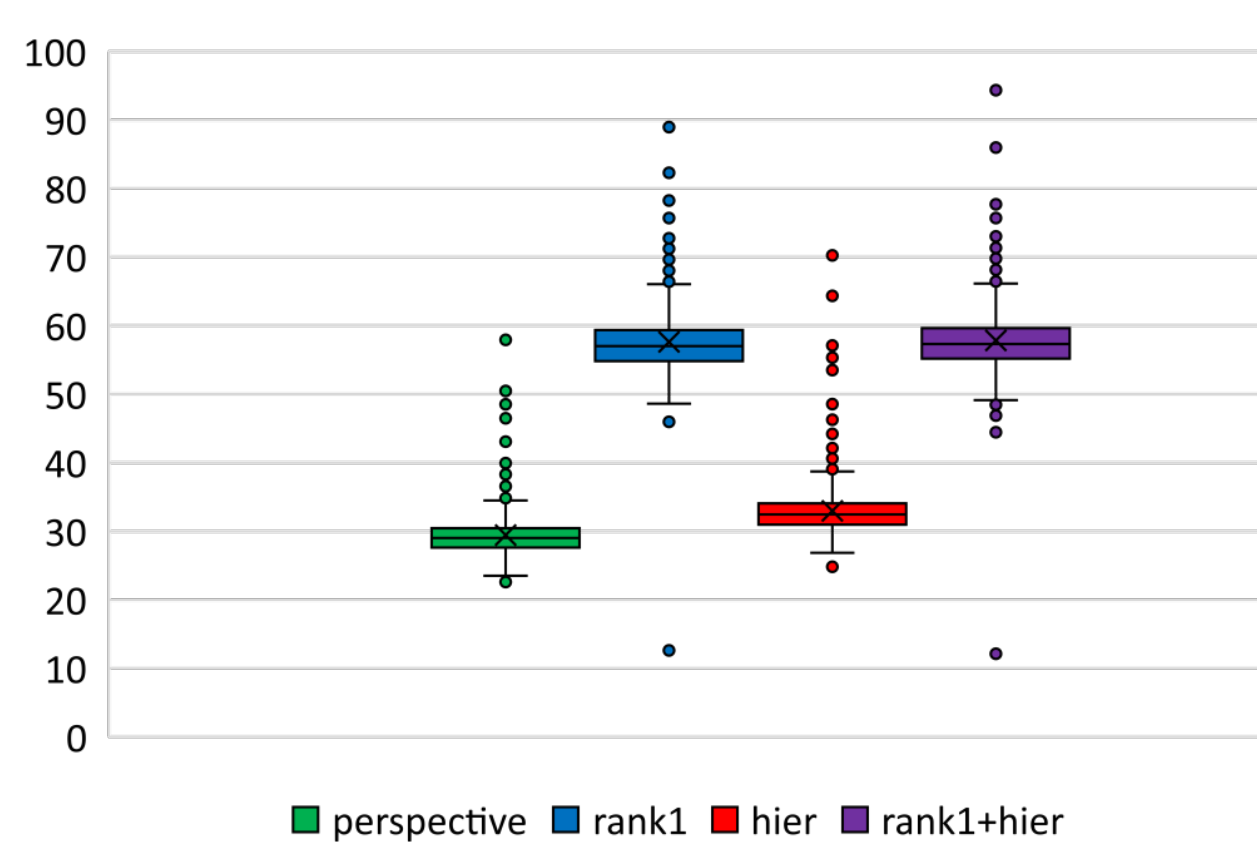}}\hfill
	\subfloat[Bias\_correction]{\includegraphics[width=0.5\textwidth,trim={10.1cm 5.5cm 11cm 5cm},clip]{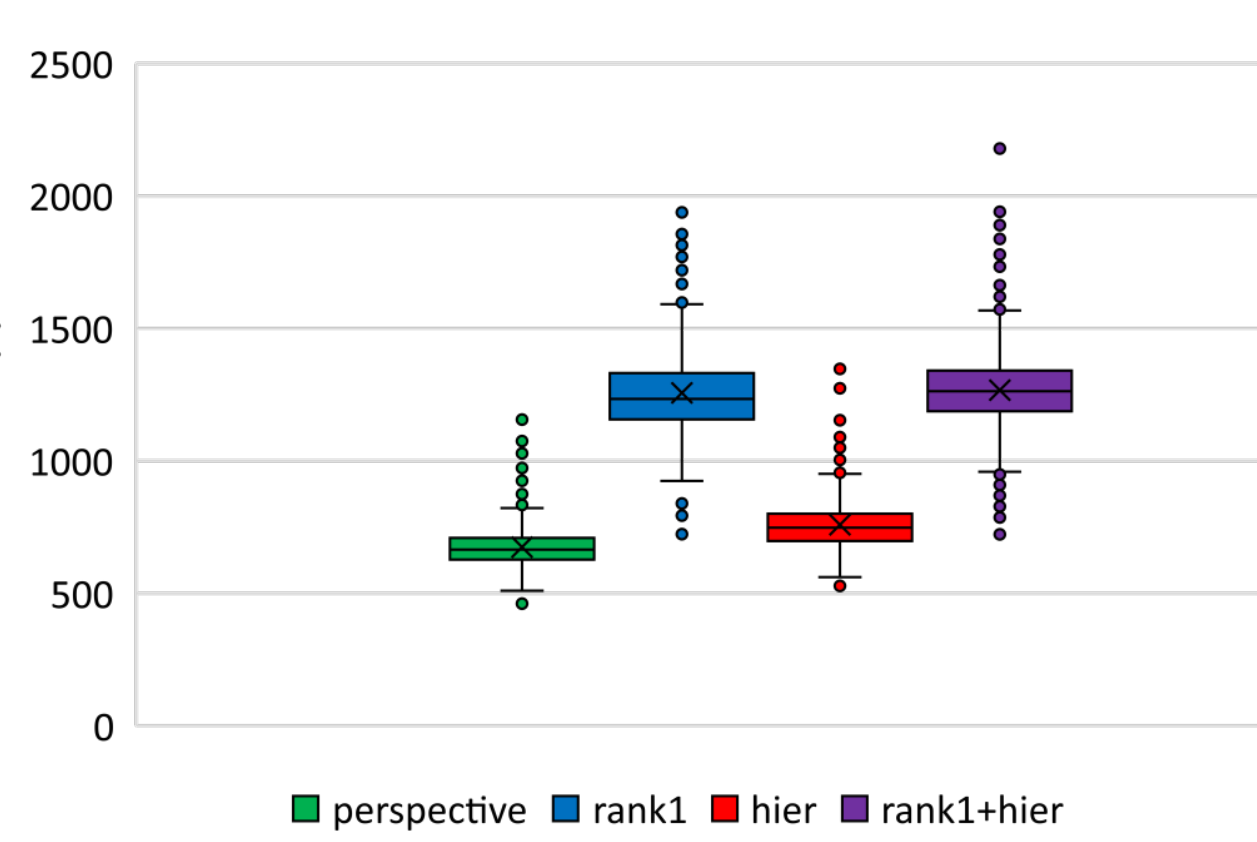}}
	\caption{Computational times in seconds.}
	\label{fig:times}
\end{figure}

Figure~\ref{fig:times} shows the distribution of times needed to solve the regression problems for each dataset. As expected, the \revised{optimal} \textbf{perspective} formulation \eqref{eq:perspective} is the fastest, as it is the simplest relaxation. We also see that formulations involving the rank-one constraints (with or without hierarchical strengthening) are  more computationally demanding, taking four times longer to solve than the perspective formulation in Crime, and twice as long in the \revised{remaining five} instances. In contrast, the formulation \textbf{Hier}, which includes hierarchical constraints but not the rank-one constraints, is much faster, requiring 70\% more time than \textbf{perspective} in the Crime dataset, and only 10\revised{-20}\% more in the other instances. Indeed, there are only $\mathcal{O}(p^2)$ hierarchical constraints to be added, while there are $\mathcal{O}(\left(p(p+3)/2\right)^2)$ rank-one constraints.

\begin{figure}[!h]
	\centering
	\subfloat[Crime]{\includegraphics[width=0.5\textwidth,trim={10.3cm 5.5cm 9cm 5cm},clip]{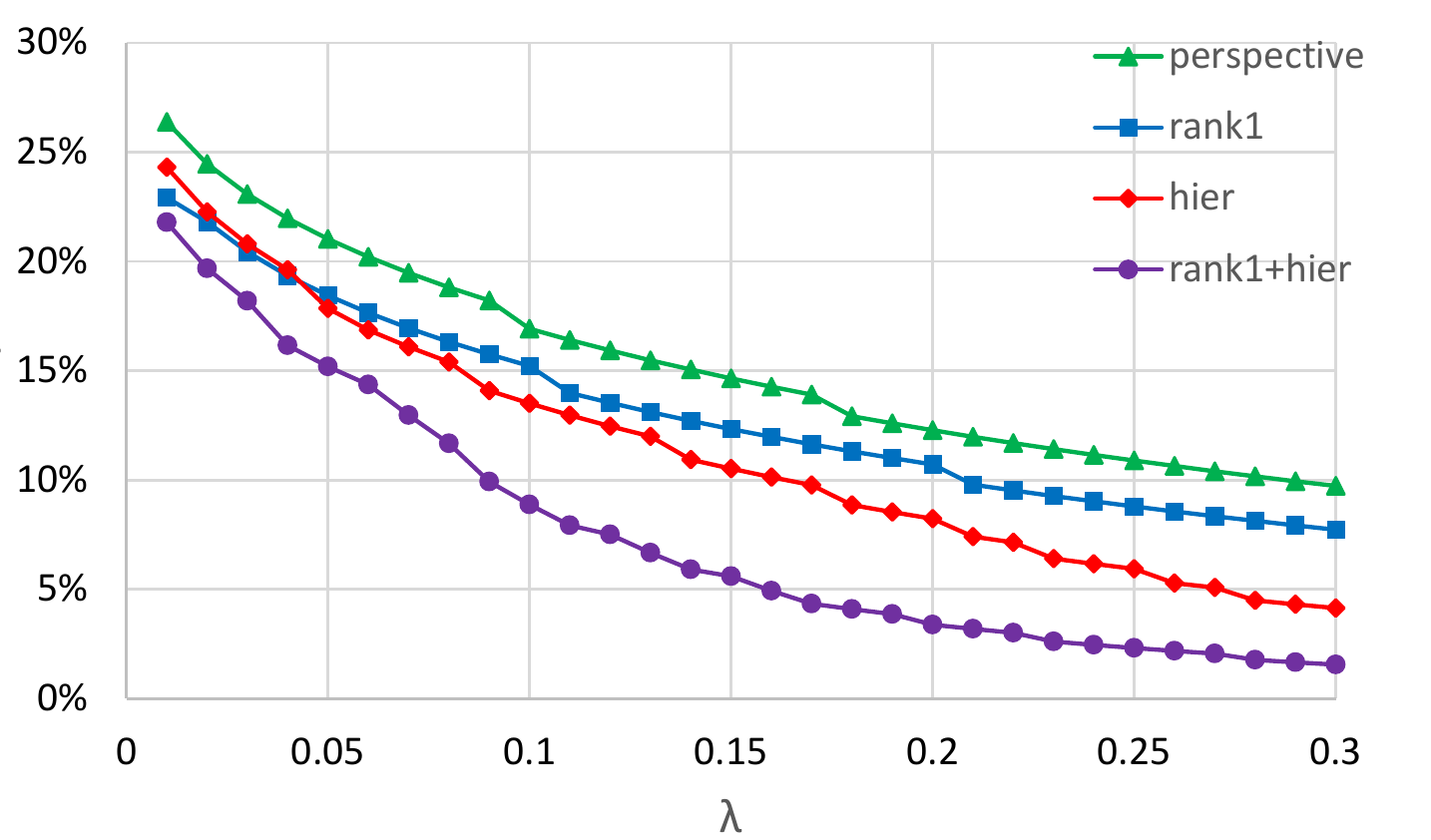}}\hfill\subfloat[Diabetes]{\includegraphics[width=0.5\textwidth,trim={10.3cm 5.5cm 9cm 5cm},clip]{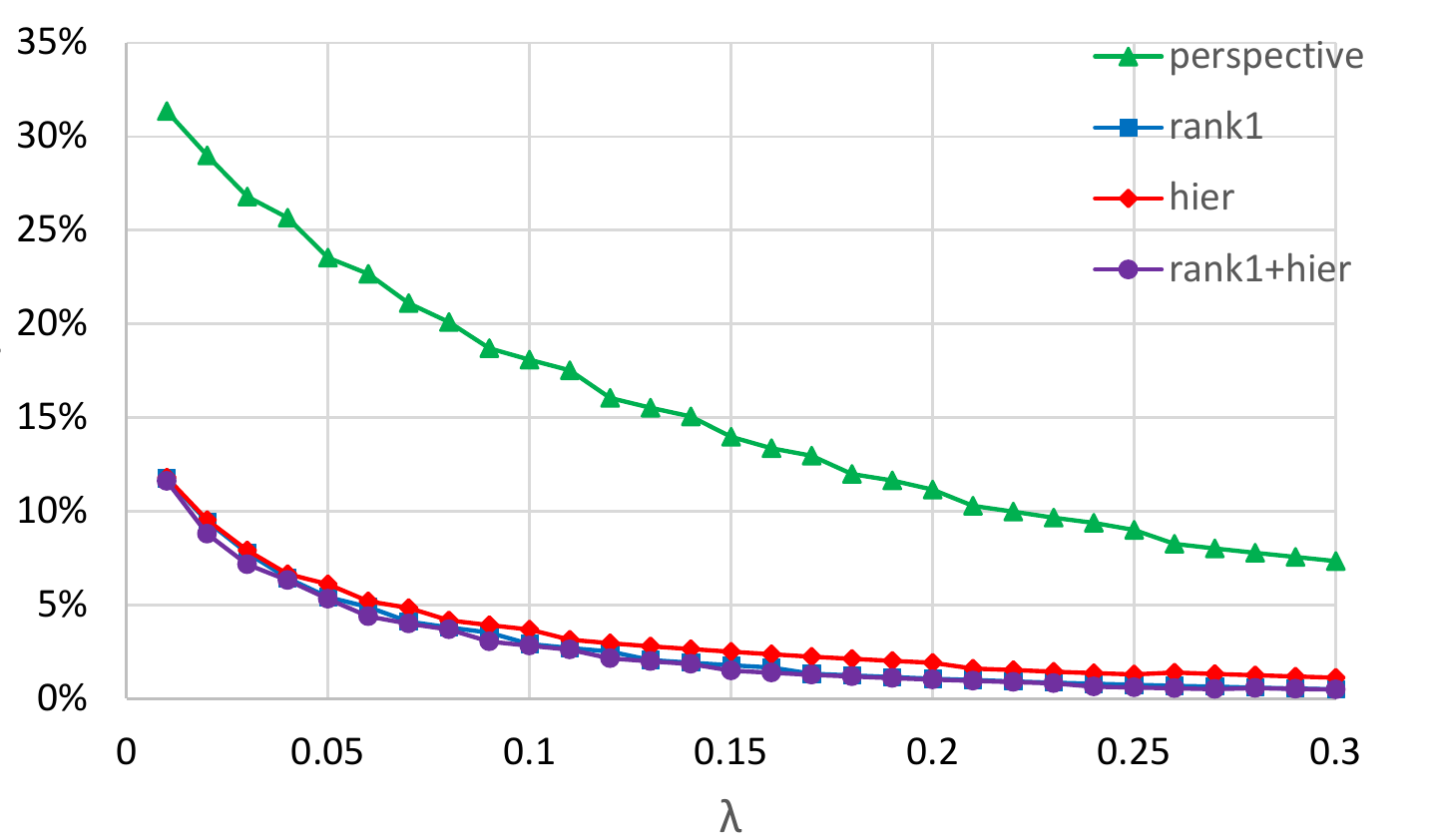}}\hfill
	\subfloat[Wine\_quality]{\includegraphics[width=0.5\textwidth,trim={10.3cm 5.5cm 9cm 5cm},clip]{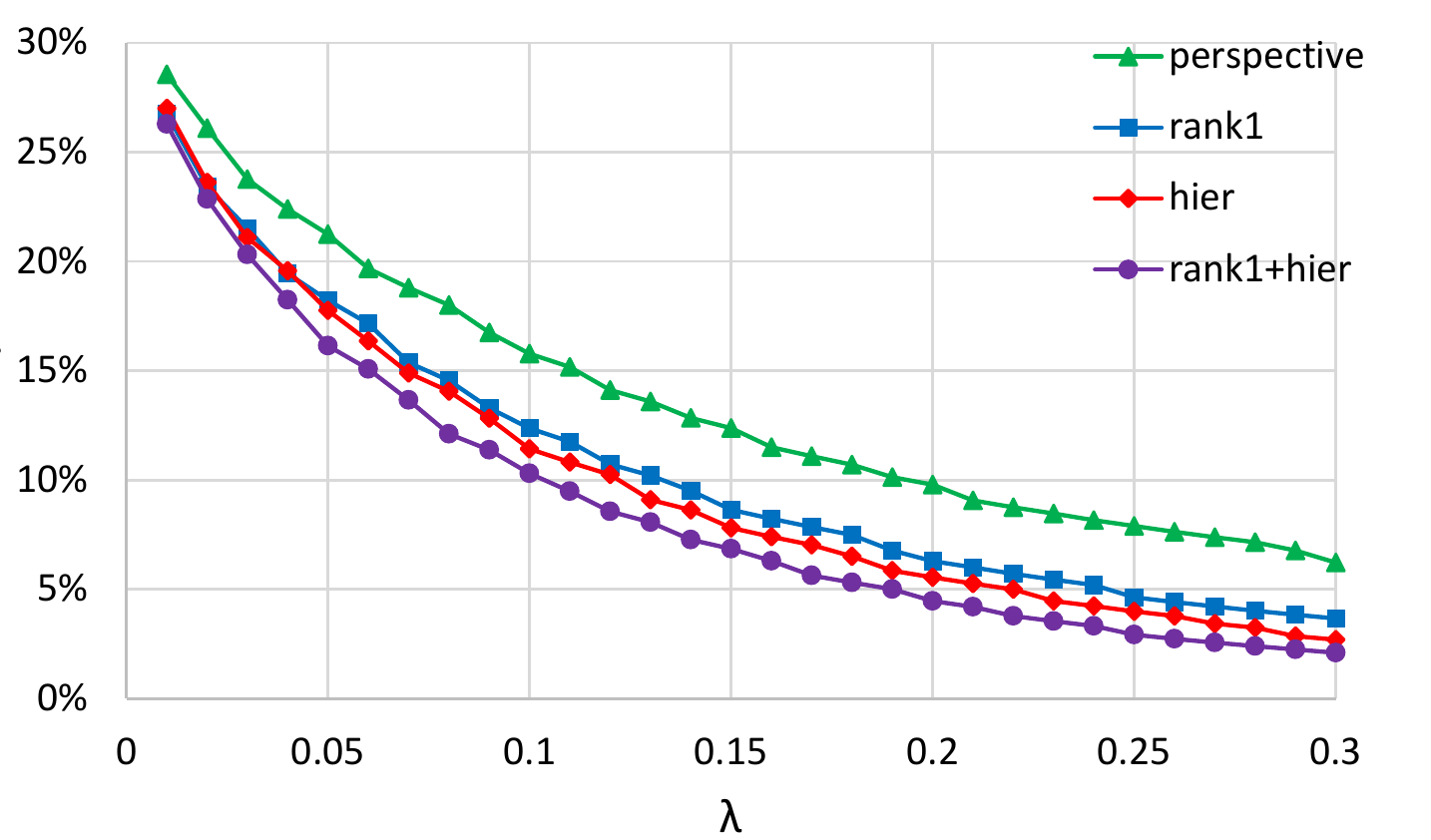}}\hfill\subfloat[Forecasting\_orders]{\includegraphics[width=0.5\textwidth,trim={10.3cm 5.5cm 9cm 5cm},clip]{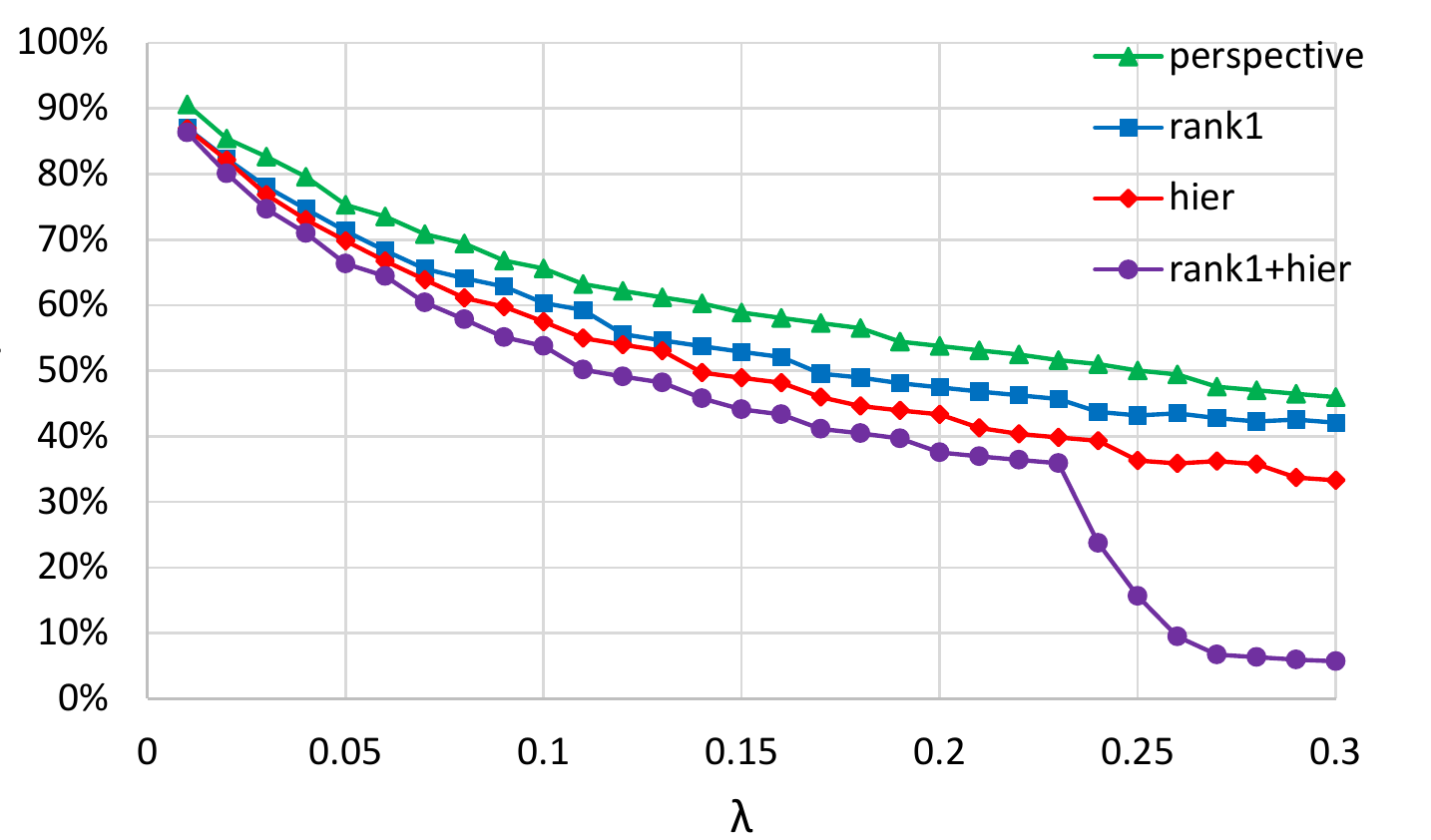}}\hfill
	\subfloat[Housing]{\includegraphics[width=0.5\textwidth,trim={10.3cm 5.5cm 9cm 5cm},clip]{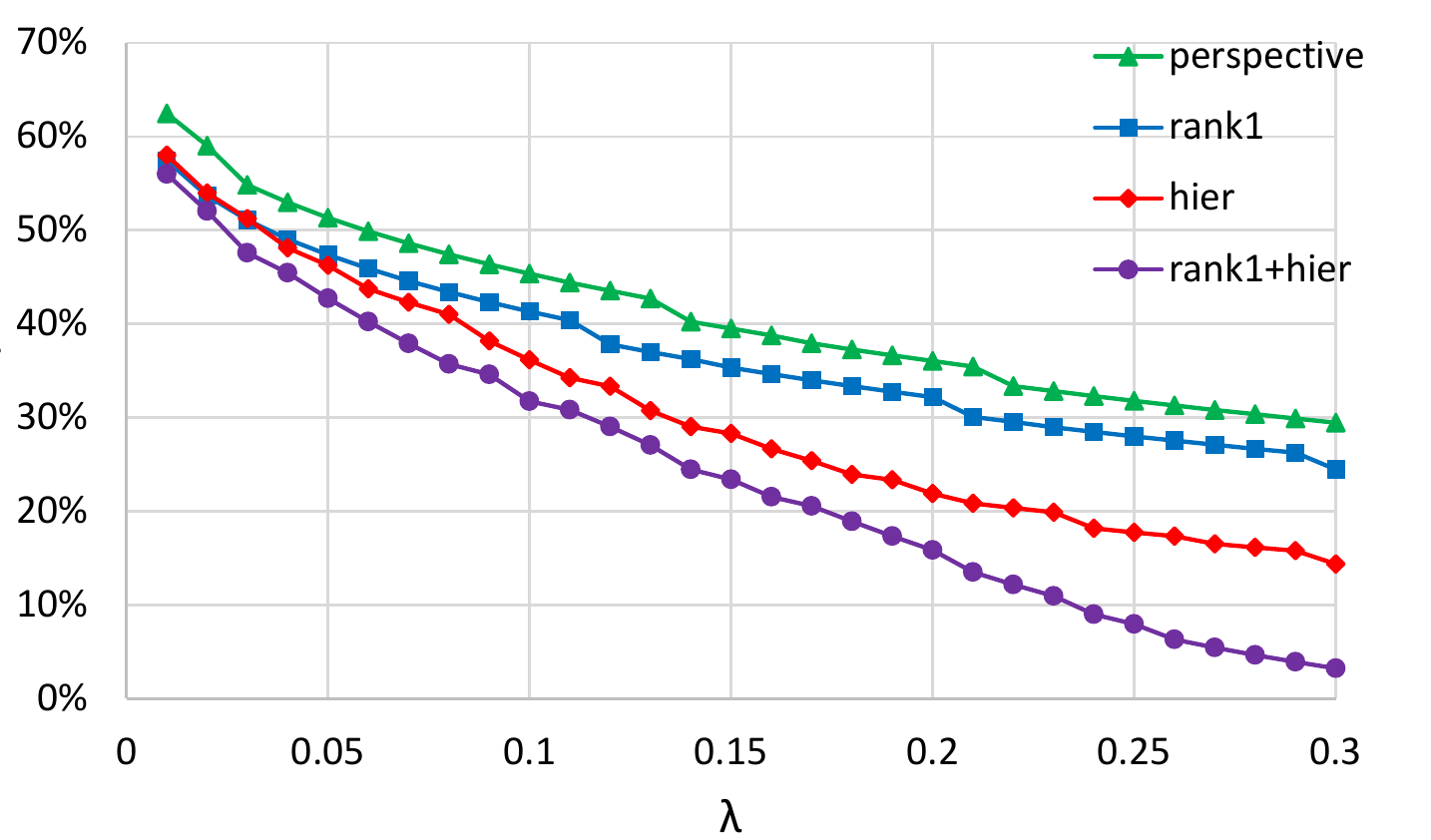}}\hfill
	\subfloat[Bias\_correction]{\includegraphics[width=0.5\textwidth,trim={10.3cm 5.5cm 9cm 5cm},clip]{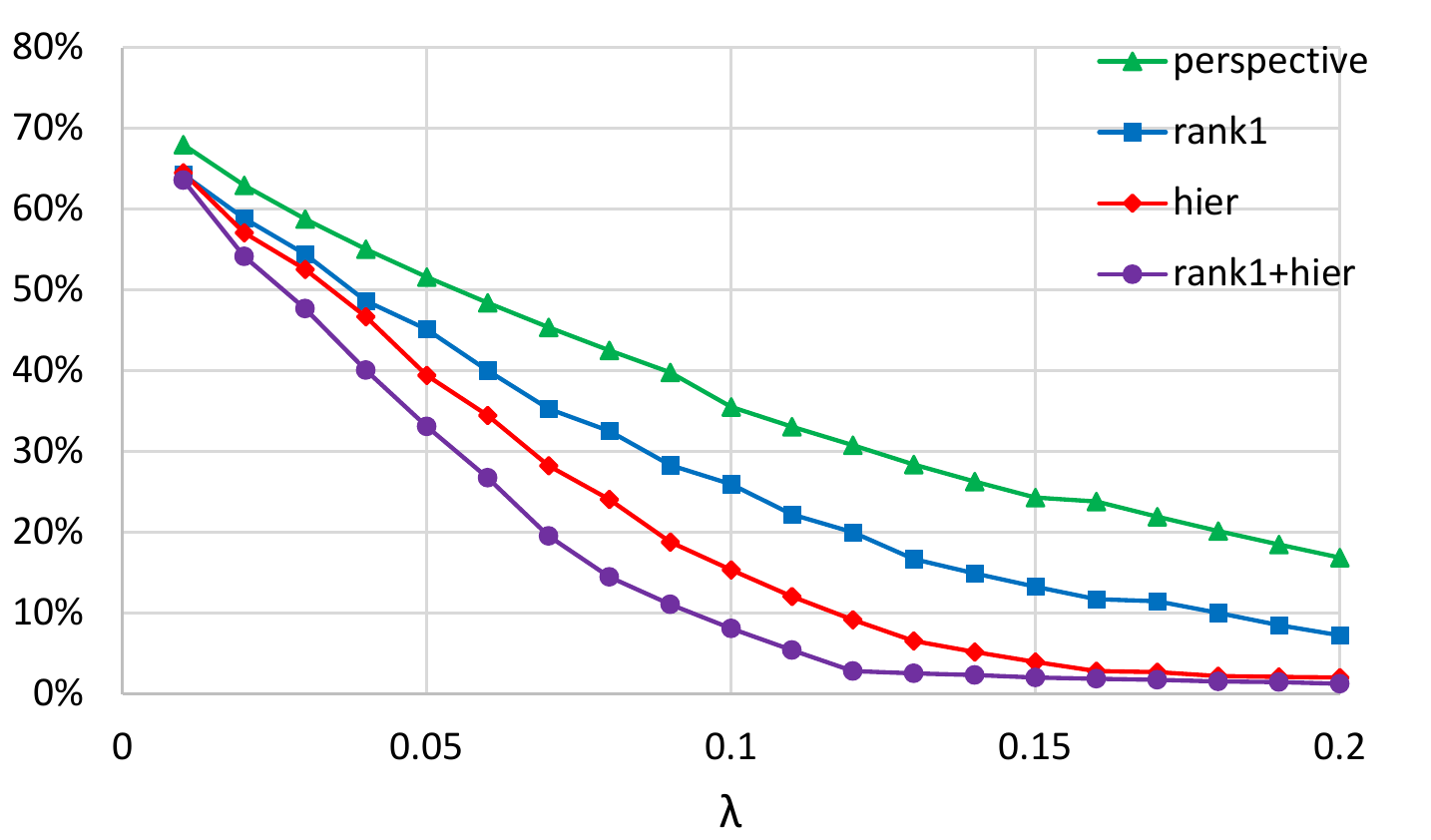}}
	\caption{Optimality gaps as a function of $\lambda$. Each point in the graph represents the average across all values of $\mu$.}
	\label{fig:lambda}
\end{figure}

Figure~\ref{fig:lambda} shows, for each dataset, the average optimality gaps as a function of the regularization parameter $\lambda$. Each point in the graph represents, for a given value of $\lambda$ the average across all 31 values of $\mu$. Similarly, Figure~\ref{fig:mu} shows, for each dataset, the average optimality gap as a function of the regularization parameter $\mu$. As expected, the optimality gaps obtained from the \revised{optimal} perspective reformulation are the largest, as the relaxation \eqref{eq:perspective} is dominated by all the other relaxations used. Moreover, the relaxation \textbf{rank1+hier} results in the smallest gaps, as it dominates every other relaxation used. 
 Finally, neither relaxation \textbf{rank1} nor \textbf{hier} consistently outperforms the other, \revised{although \textbf{hier} results in lower gaps overall in all datasets except  Diabetes.}

\begin{figure}[!h]
	\centering
	\subfloat[Crime]{\includegraphics[width=0.5\textwidth,trim={10.3cm 5.5cm 9cm 5cm},clip]{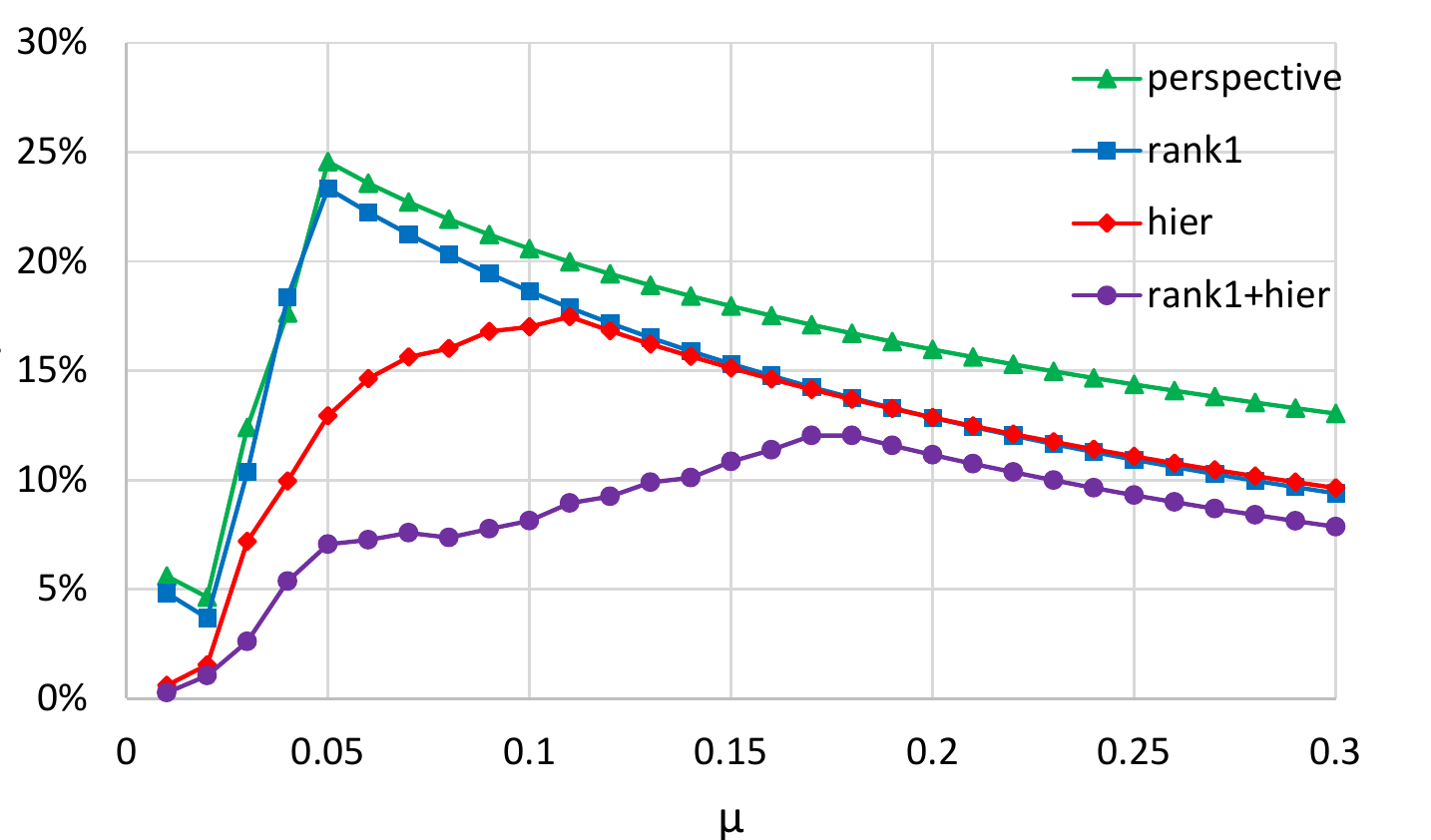}}\hfill\subfloat[Diabetes]{\includegraphics[width=0.5\textwidth,trim={10.3cm 5.5cm 9cm 5cm},clip]{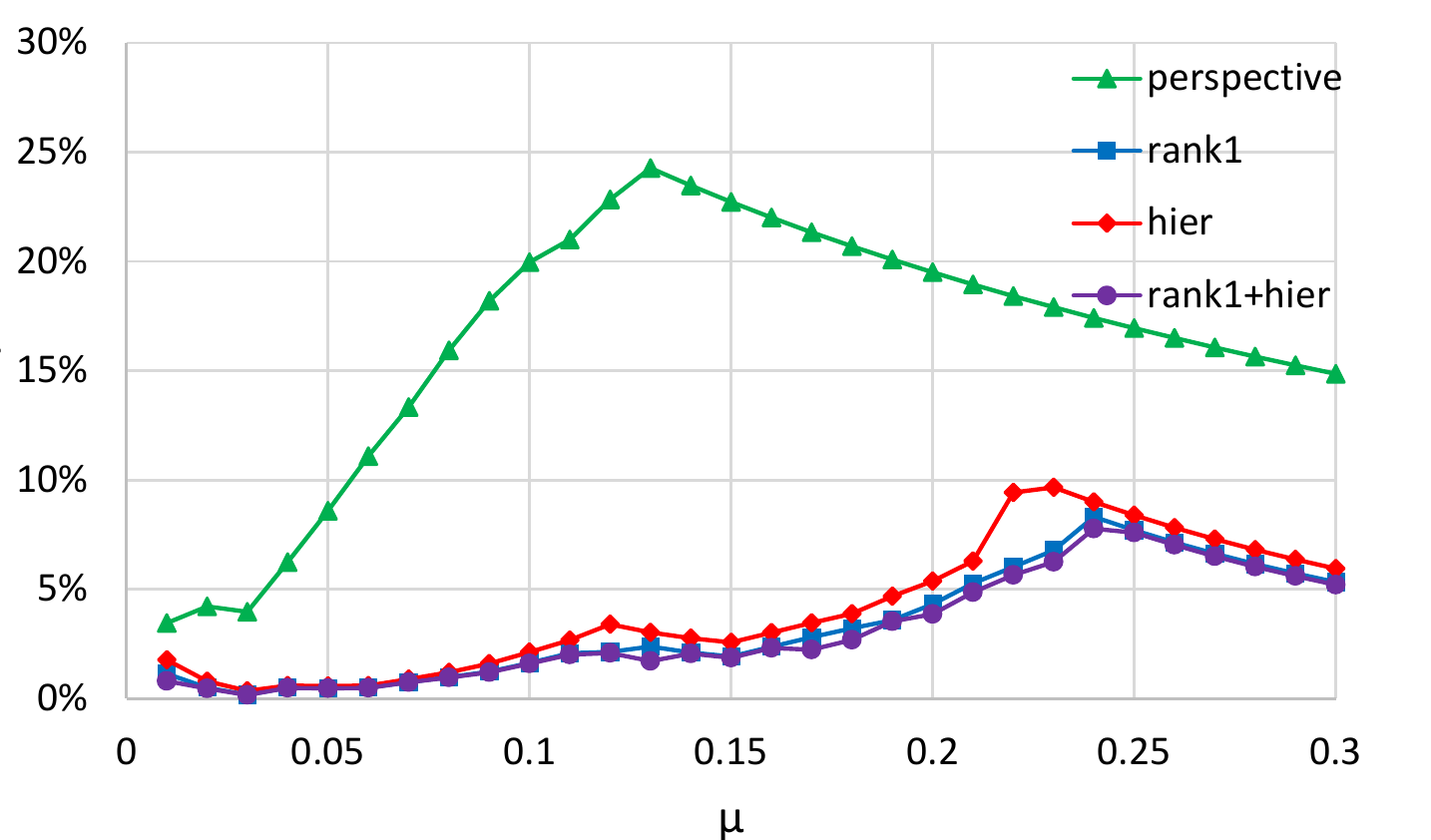}}\hfill
	\subfloat[Wine\_quality]{\includegraphics[width=0.5\textwidth,trim={10.3cm 5.5cm 9cm 5cm},clip]{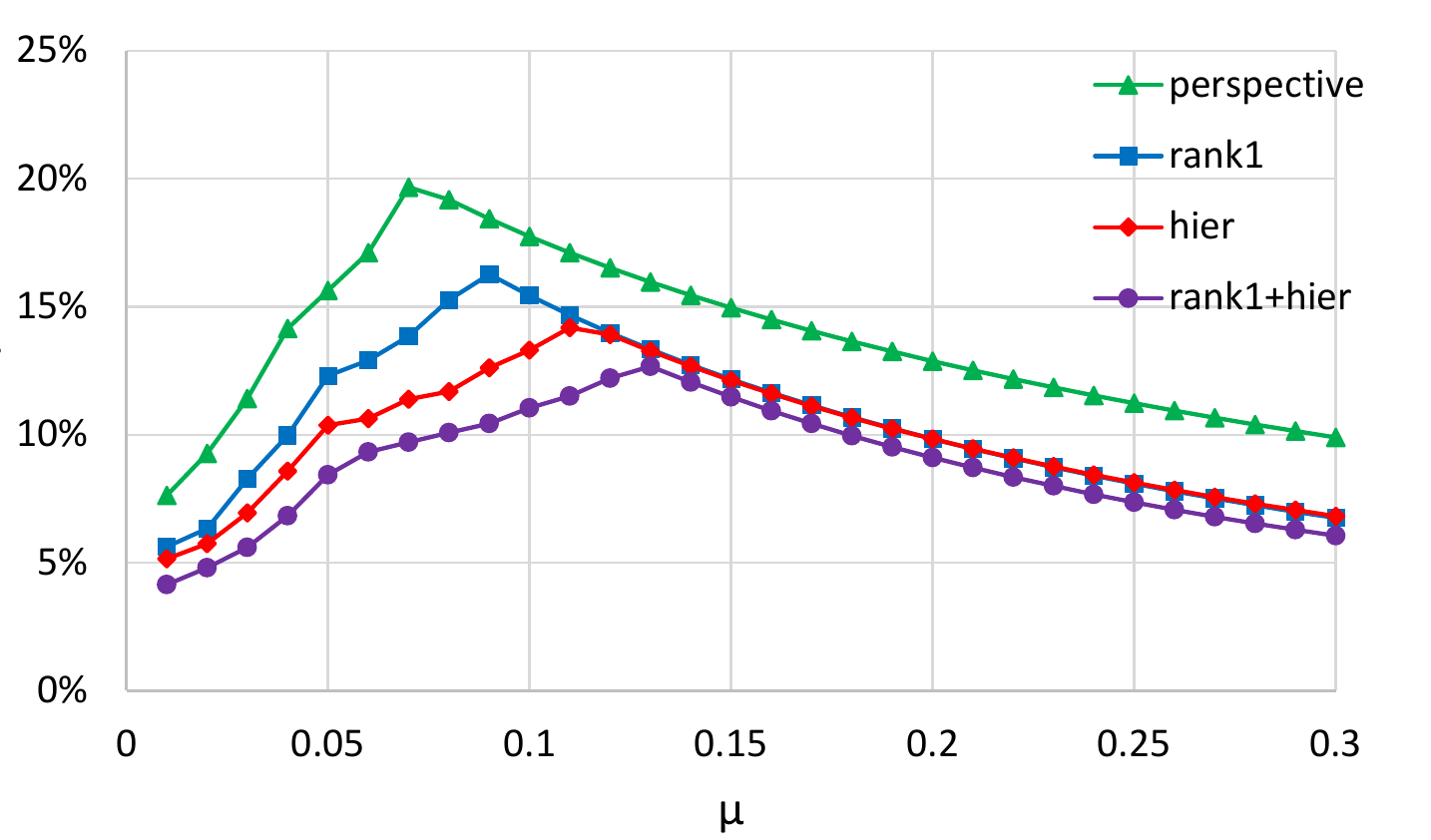}}\hfill\subfloat[Forecasting\_orders]{\includegraphics[width=0.5\textwidth,trim={10.3cm 5.5cm 9cm 5cm},clip]{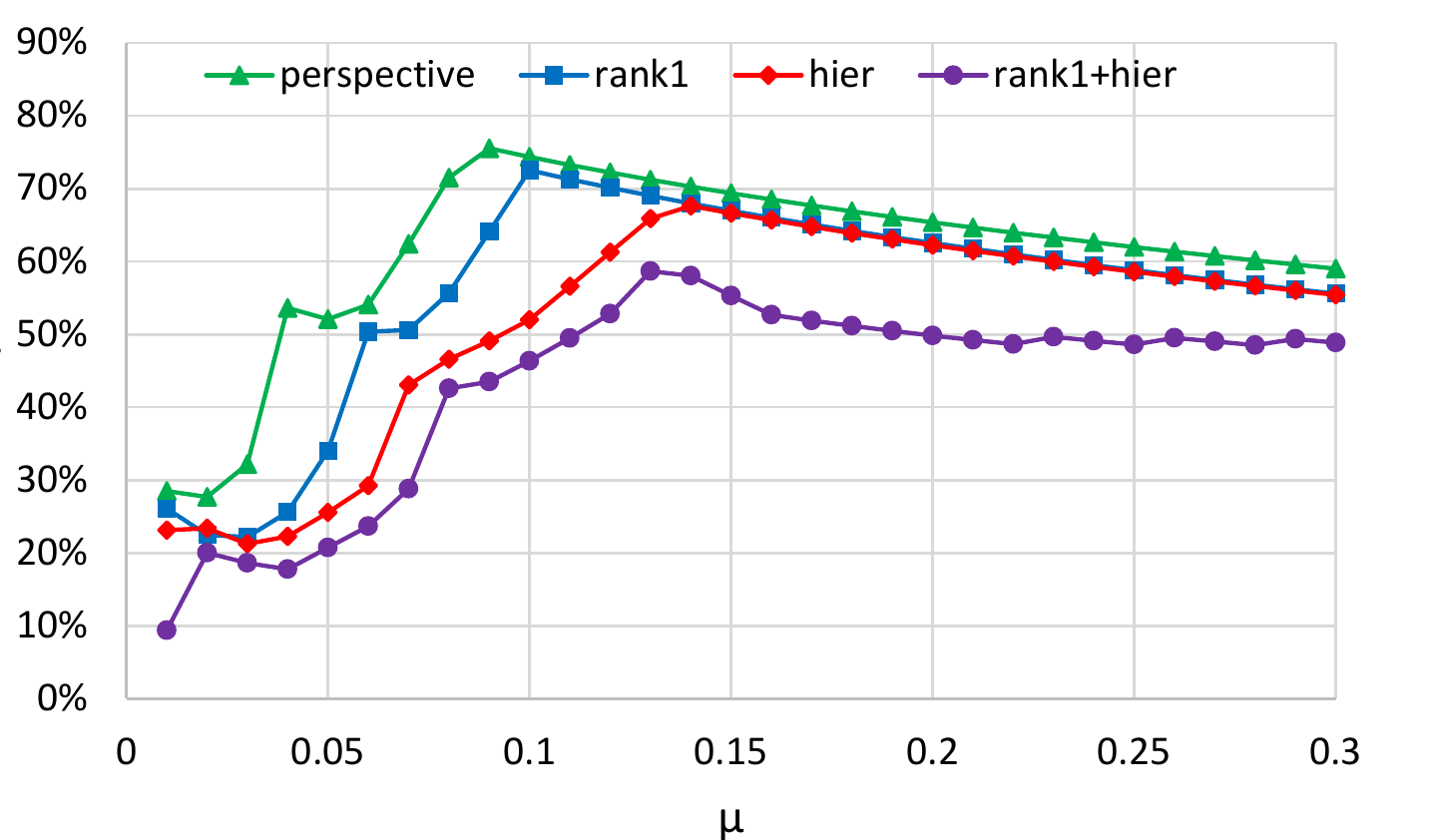}}\hfill
	\subfloat[Housing]{\includegraphics[width=0.5\textwidth,trim={10.3cm 5.5cm 9cm 5cm},clip]{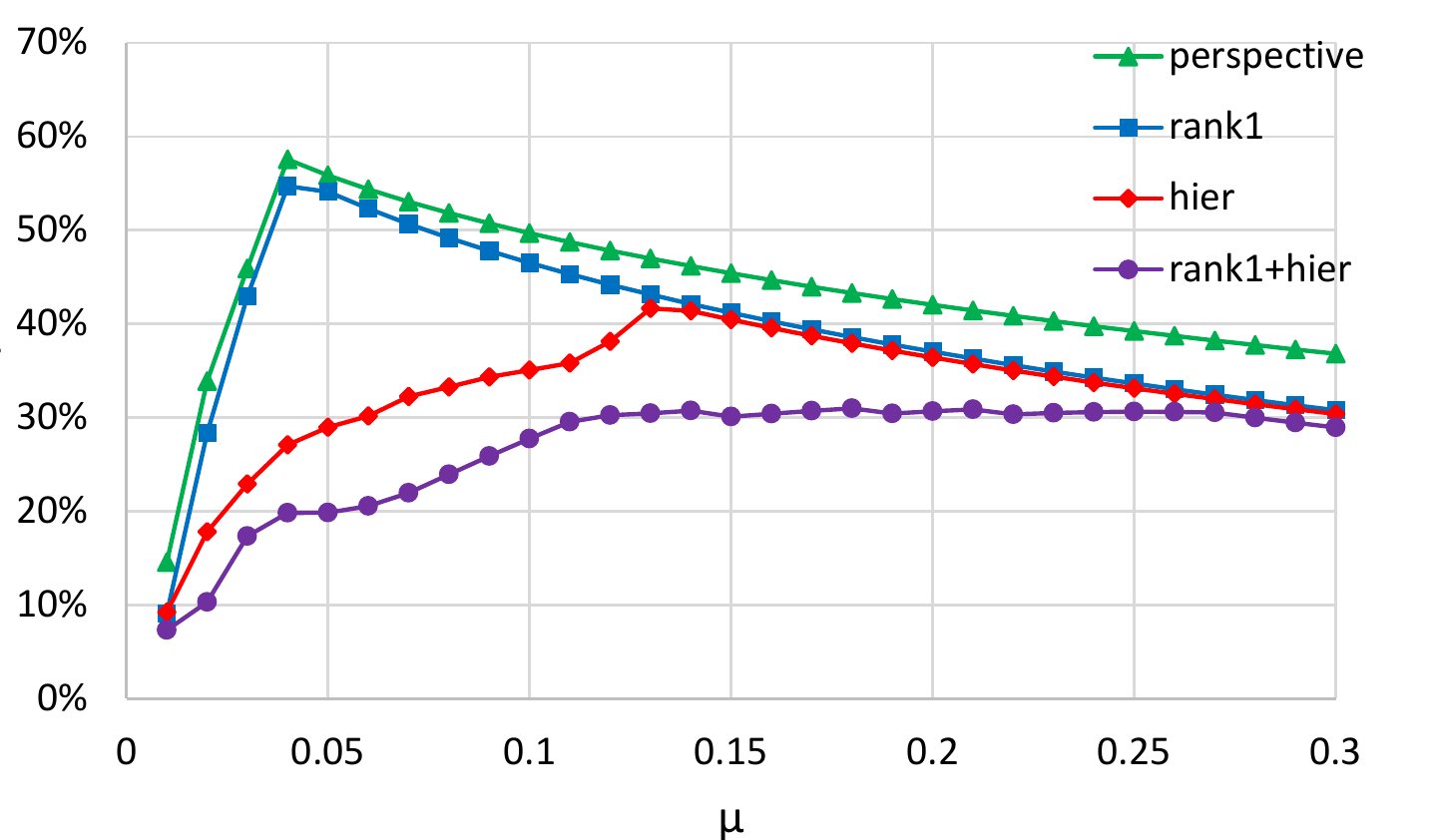}}
	\hfill
	\subfloat[Bias\_correction]{\includegraphics[width=0.5\textwidth,trim={10.3cm 5.5cm 9cm 5cm},clip]{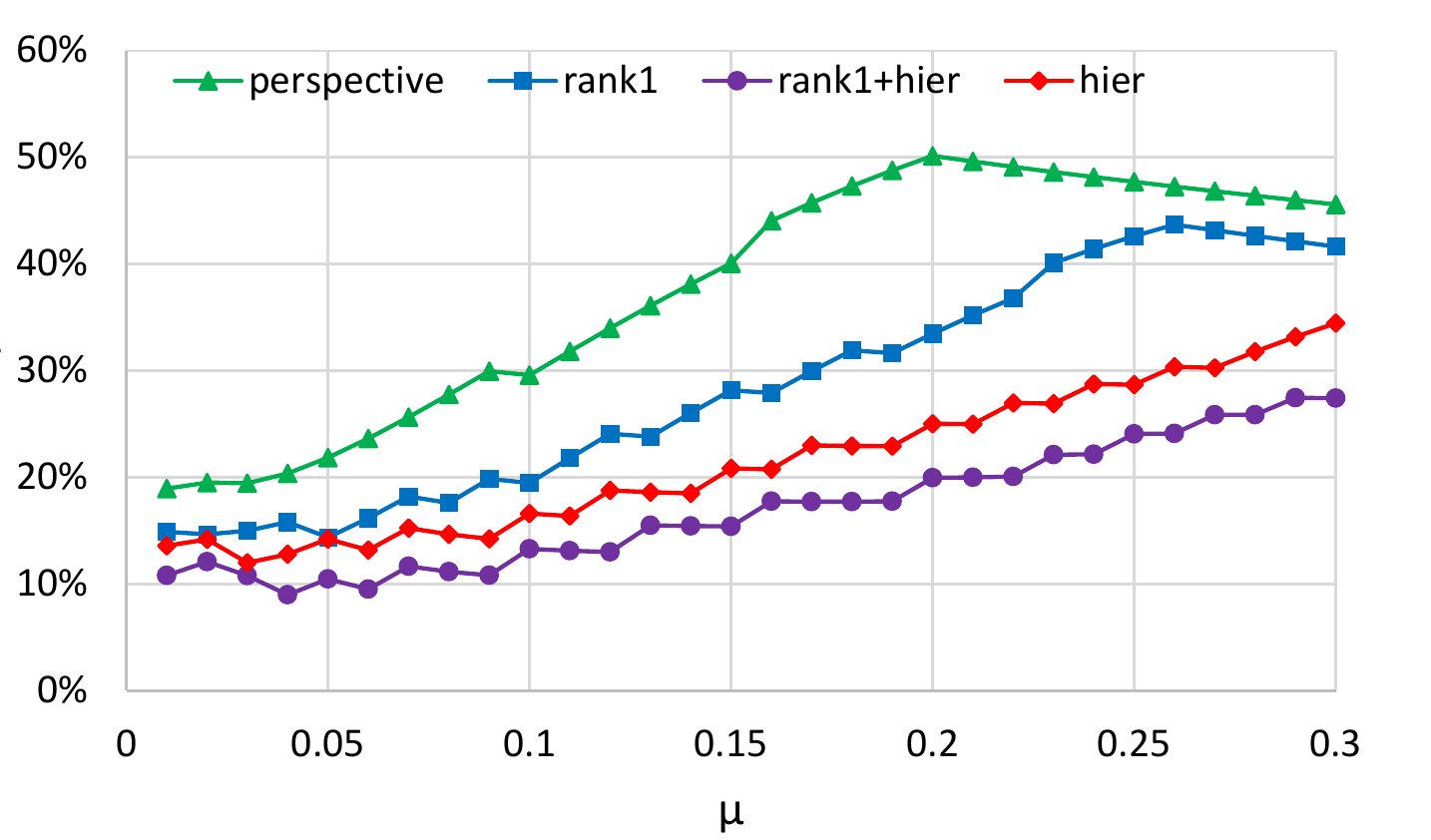}}
	\caption{Optimality gaps as a function of $\mu$. Each point in the graph represents the average across all values of $\lambda$.}
	\label{fig:mu}
\end{figure}

The relative performance of the formulations tested in terms of gap largely depends on the dataset and parameters used. In the Diabetes \revised{and Wine\_quality} datasets, the \textbf{perspective} reformulation is by far the worst, and all other formulations significantly improve upon it. Specifically, \textbf{rank1+hier} is slightly better than \textbf{rank1} \revised{and \textbf{hier} (which have similar strengths)}, but the differences are marginal---observe that \textbf{hier} achieves an almost ideal strengthening with half the computational cost of the other formulations. In contrast, in the Crime, Housing\revised{ and Bias\_correction} datasets, \textbf{rank1} achieves only a marginal improvement over the \textbf{perspective} relaxation, while \textbf{hier} achieves a significant improvement over \textbf{rank1}, and \textbf{rank1+hier} results in a even more substantial improvement. For example, for the Housing dataset, for $\lambda=0.3$ the average  optimality gap of \textbf{perspective} is 29\%, whereas that of  \textbf{rank1+hier}  is 3\%. \revised{Finally, in Forecasting\_orders, all formulations perform similarly for $\lambda\leq 0.23$; however, for $\lambda\geq 0.24$, \textbf{rank1+hier} results in a significant improvement over the other formulations. Note that Forecasting\_orders is a ``fat" dataset with $n<p(p+3)/2$, which is more difficult for convexifications of the form \eqref{eq:perspective} for low values of $\lambda$.  Our conclusions from Figure \ref{fig:mu} are similar. Of particular note is the marked improvement in the optimality gaps for  \textbf{rank1+hier}  over other formulations (especially  \textbf{perspective}) for small $\mu$.  For example, for $\mu=0.05$,  the average optimality gap of   \textbf{rank1+hier} in the  Crime dataset is slightly over 5\%, whereas \textbf{hier} achieves over 10\% gap, and  \textbf{perspective}  and \textbf{rank1}  result in  25\% gap.  
}

Since \textbf{hier} has a similar computational cost as \textbf{perspective}, and \textbf{rank1+hier} has a virtually identical cost as \textbf{rank1}, we see that the hierarchical strengthening may lead to large improvements without drawbacks (whereas \textbf{rank1} requires 2--4 times more computational overhead). Indeed, the hierarchical strengthening is tailored to problem \eqref{eq:hierarchy}, while \textbf{rank1} is more general but does not exploit any structural information from the constraints. 

\rev{
\begin{remark}[On rounding vs mixed-integer optimization] An alternative to the rounding approach used here is to use mixed-integer optimization (MIO) to solve \eqref{eq:hierarchy} exactly. An extensive comparison between
MIO and the \textbf{rank1} approach was performed in \cite{atamturk2019rank} in a variety of real datasets, including the Diabetes dataset used here. In summary, while MIO (using the perspective reformulation) was found to be more effective for large values of the parameter $\lambda$, simple rounding of the \textbf{rank1} relaxation was already sufficient to prove smaller optimality gaps than MIO if $\lambda$ is small. We refer to reader to \cite{atamturk2019rank} for additional details. 
\end{remark}
}

\revised{
\subsection{Sparse Logistic Regression}\label{sec:logistic}

To illustrate the convexification of non-quadratic functions derived in \S\ref{sec:convexification}, we consider $\ell_0$-regularized logistic regression problems. Specifically, given a classification problem with data $(x_i,y_i)_{i=1}^n$ where $x_i\in \R^p$ and $y_i\in\{-1,1\}$, sparse logistic regression calls for solving the problem \cite{dedieu2020learning,sato2016feature}
\begin{subequations}\label{eq:logMIP}
\begin{align}
	\min_{z,\beta} \quad & (1 - \lambda)\sum_{i = 1}^{n} \log\left(1+ \exp(-y_ix_i^\top\beta) \right) + \lambda \sum_{i = 1}^{p} z_i \\
	\text{s.t.} \quad & \beta_i (1 - z_i) = 0,\quad \quad \forall i\in [p],  
\end{align}
\end{subequations}
where $0\leq \lambda \leq 1$ is a regularization coefficient that controls the balance between the error and the $\ell_0$-penalty. Note that the natural convex relaxation of \eqref{eq:logMIP}, obtained by dropping the indicator variables $z$ (or, equivalently, adding big M constraints $|\beta_i|\leq Mz_i$ with $M\to \infty$), is convex. 

To date, there are limited results concerning convexifications of \eqref{eq:logMIP}, especially when compared with the sparse least squares regression problem \eqref{eq:leastsquares}, due to: \textit{(i)} non-existence of separable terms amenable to the perspective relaxation; \textit{(ii)} lack of convexifications for non-separable non-quadratic terms with indicators; and \textit{(iii)} non-decomposability of the objective function into simpler terms, resulting in similar convexifications as those discussed in \S\ref{sec:quadratic}. In this paper we provided the first convexifications for non-quadratic non-separable functions, addressing issue \textit{(ii)}. In this section we illustrate that \emph{if} the observations $x_i$ are sufficiently sparse, then a direct application of Theorem~\ref{theo:hullConnected} results in substantial improvements over the natural relaxation, circumventing issues   \emph{(i)} and \emph{(iii)}.

\subsubsection{Formulations}

A direct application of Theorem~\ref{theo:hullConnected}, corresponding to strengthening each error term $\log\left(1+ \exp(-y_ix_i^\top\beta) \right)\leq t_i$ individually, yields the following ``rank-one" relaxation of the sparse logistic regression problem \eqref{eq:logMIP}:
\begin{subequations}\label{eq:logrel}
	\begin{align}
		\min_{z,\beta,t} \quad & (1 - \lambda)\left(\sum_{i = 1}^{n} t_i + n\log(2)\right) + \lambda \sum_{i = 1}^{p} z_i  \\
		\text{s.t.} \quad & t_i \geq \log(1 + \exp (- y_i x_i^\top \beta)) - \log(2), \quad i \in [n] \label{eq:constraintconic1}\\
	(\textbf{log-rko})\qquad \qquad	& t_i \geq \left(\sum_{j:x_{ij} \neq 0} z_j\right) \log\left(1 + \exp\left(- \frac{ y_i x_i^\top\beta}{\sum_{j:x_{ij} \neq 0} z_j}\right)\right) \notag\\
		&\hspace{2cm}- \left(\sum_{j:x_{ij} \neq 0} z_j\right) \log(2), \quad \hfill i \in [n]  \label{constraintconic2}\\
		& z_i \in [0,1], \quad i \in [n]. 
	\end{align}
\end{subequations}
We can write \eqref{eq:logrel} as a conic optimization problem using the exponential cone
\[
K_{exp} = \{(w_1, w_2, w_3) \; | \; w_1 \geq w_2 e^{w_3 / w_2}, w_2 >0\} \bigcup \{(w_1, 0, w_3) \; | \; w_1 \geq 0, w_3 \leq 0\},
\] 
i.e., the closure of the set of points satisfying $w_1\geq w_2e^{w_3/w_2}$ and $w_1,w_2\geq 0$. 

\noindent Constraint \eqref{eq:constraintconic1} is equivalent to $\exists u^{i}_1, v^{i}_1$ such that :
\begin{align*}
	u^{i}_{1} + v^{i}_{1} \leq 2, \\
	(u^{i}_{1}, 1, - y_i x_i^\top\beta- t_i) \in K_{exp},\\
	(v^{i}_{1}, 1, -t_i) \in K_{exp}. 
\end{align*}
\noindent Similarly,  constraint \eqref{constraintconic2} is equivalent to $\exists u^{i}_2, v^{i}_2$ such that:
\begin{align*}
	u^{i}_{2} + v^{i}_{2} \leq 2(\sum_{j:x_{ij} \neq 0} z_j), \\
	(u^{i}_{2}, \sum_{j:x_{ij} \neq 0} z_j, - y_i x_i^\top \beta - t_i) \in K_{exp}, \\
	(v^{i}_{2}, \sum_{j:x_{ij} \neq 0} z_j, - t_i) \in K_{exp}. 
\end{align*}

We refer to formulation \eqref{eq:logrel} as \textbf{log-rko} in the sequel. We compare it with the natural convex relaxation \eqref{eq:logMIP} \textbf{log-nat}, corresponding to dropping contraints \eqref{constraintconic2} from the formulation. Observe that for relaxation \textbf{log-nat}, $z=\mathbf 0$ in an optimal solution, thus resulting in the same objective value for all values of $\lambda$. 

\subsubsection{Lower bounds}

We report lower bounds found from solving convex relaxations of \eqref{eq:logMIP}.  The optimal values of the relaxations considered are divided by $(1 - \lambda)n \log(2)$. Thus the feasible solution of \eqref{eq:logMIP} obtained by setting $z=\beta=\mathbf 0$ has objective value 1 (this solution may be optimal for large values of $\lambda$). The objective value of \eqref{eq:logMIP} also have a trivial lower bound of $0$, attained if the data can be perfectly classified (observe that if $n<p$ and $\lambda\to 0$, this lower bound may in fact be attained).

\subsubsection{Instances and parameters}
For the synthetic datasets, we consider the case where both the input data and the true model are sparse. Let  $\alpha$ be  a parameter that controls the sparsity of features $(x_i)_{i=1}^n$. For each entry $x_{ij}$ we either independently assign a value of zero with probability $1-\alpha$ or we sample from a standard normal distribution $\mathcal N (0,1)$ with probability $\alpha$.  We generate a ``true" sparse coefficient vector $\beta^*$ with $s$ uniformly sampled non-zero indices such that $\beta^*_i = 1$. The responses $y_i \in \{-1, 1\}$ are then generated independently from a Bernoulli distribution with: $P(y_i = 1 | x_i) = (1 + \exp{(- x_i^\top \beta^*)})^{-1}$. We use regularization values $\lambda \in \{0.1, 0.3, 0.5, 0.7\}$ and sparsity levels $\alpha \in \{0.01, 0.02, 0.05, 0.1\}$. Moreover, we set $p=100$, $s=1$ and test varying sample sizes $n = 50, 100, 200$.

\begin{table*}\centering
	\ra{1.3}
	\caption{Scaled lower bounds for varying $n$, $\lambda$, and $\alpha$. Each entry in the table represents the average across ten random instances.}
	\label{tab:LB}
	\begin{tabular}{@{} rrrrrr@{}}\toprule
		&& $\alpha = 0.01$ & $\alpha = 0.02$ & $\alpha = 0.05$ & $\alpha = 0.1$\\ \midrule
		$n = 50$\\
		&\textbf{log-nat\hfill} & 0.430 & 0.136 & 0.008 & 0 \\ 
		&\textbf{log-rko} $\lambda = 0.1$ & 0.502 & 0.214 & 0.059 & 0.027\\
		&$\lambda = 0.3$ & 0.703 & 0.434 & 0.203 & 0.103\\
		&$\lambda = 0.5$ & 0.92 & 0.726 & 0.441 & 0.239\\
		&$\lambda = 0.7$ & 1 & 0.965 & 0.798 & 0.524\\
		$n = 100$\\
		&\textbf{log-nat\hfill} & 0.392 & 0.164 & 0.003 & 0 \\
		
		&\textbf{log-rko} $\lambda = 0.1$ & 0.454 & 0.221 & 0.036 & 0.017\\
		&$\lambda = 0.3$ & 0.626 & 0.383 & 0.133 & 0.066\\
		&$\lambda = 0.5$ & 0.834 & 0.615 & 0.296 & 0.153\\
		&$\lambda = 0.7$ & 0.985 & 0.893 & 0.588 & 0.340\\
		$n = 200$ \\
		&\textbf{log-nat \hfill} & 0.519 & 0.328 & 0.220 & 0.232\\
		&\textbf{log-rko} $\lambda = 0.1$ & 0.564 & 0.370 & 0.243 & 0.242\\
		&$\lambda = 0.3$ & 0.685 & 0.483 & 0.308 & 0.272\\
		&$\lambda = 0.5$ & 0.837 & 0.644 & 0.414 & 0.323\\
		&$\lambda = 0.7$ & 0.974 & 0.859 & 0.598 & 0.434\\
		\bottomrule
	\end{tabular}
\end{table*}

\subsubsection{Results}
Table~\ref{tab:LB} shows the scaled lower bounds obtained via convex relaxations \textbf{log-nat} and \textbf{log-rko}.
Each entry in the table corresponds to the average (over ten replications) lower bound obtained from a given relaxation for a particular combination of sparsity level $\alpha,\lambda$, and number of observations $n$. Recall that \textbf{log-nat} results in the same objective regardless of the value of $\lambda$.

\noindent Compared with the natural relaxation of sparse logistic regression, the lower bound attained by  \eqref{eq:logrel} increases significantly when $n \leq p$.  Moreover, as expected, larger improvements of \textbf{log-rko} over \textbf{log-nat} are obtained for larger values of $\lambda$, where sparsity plays a more prominent role in the objective value. The lower bounds of \textbf{log-rko} are at least $16\%$ more than those of \textbf{log-nat} in all test cases, and sometimes substantially larger (e.g., in cases \textbf{log-nat} results in the trivial lower bound of $0$). When the input data is very sparse, i.e., $\alpha=0.01$ and $\lambda \ge 0.5$, \textbf{log-rko} results in lower bounds close to $1$ or equal to $1$, suggesting (and in some cases proving) that true optimal solution in those cases is $z=\beta=\mathbf 0$. When $n > p$, \textbf{log-rko} still results in better lower bounds than \textbf{log-nat}, although improvements are less pronounced in these cases.  
}

\section{Conclusions}
In this paper, we propose a unifying convexification technique for the epigraphs of a class of convex functions with indicator variables constrained to certain polyhedral sets. We illustrate the utility of our approaches on constrained regression problems of recent interest. Our results generalize the existing results that consider only quadratic, separable or differentiable convex functions, and certain structural constraints such as cardinality or unit commitment.
 As future research, we plan to consider convexifications for more general  functions. 

\section*{Acknowledgments}

We thank the AE and two referees whose comments expanded  and improved our computational study, and also led to the result  in Appendix \ref{sec:app}. This research is supported, in part, by ONR grant N00014-19-1-2321, and NSF grants   1818700, 2006762, and 2007814. A preliminary version of this work appeared in \citet{WGK20}.

\bibliographystyle{apalike}
\bibliography{Bibliography}

\begin{thebibliography}{}

\bibitem[Akt{\"u}rk et~al., 2009]{akturk2009strong}
Akt{\"u}rk, M.~S., Atamt{\"u}rk, A., and G{\"u}rel, S. (2009).
\newblock A strong conic quadratic reformulation for machine-job assignment
  with controllable processing times.
\newblock {\em Operations Research Letters}, 37(3):187--191.

\bibitem[Angulo et~al., 2015]{forbidden2015}
Angulo, G., Ahmed, S., Dey, S.~S., and Kaibel, V. (2015).
\newblock Forbidden vertices.
\newblock {\em Mathematics of Operations Research}, 40(2):350--360.

\bibitem[Anstreicher, 2012]{anstreicher2012convex}
Anstreicher, K.~M. (2012).
\newblock On convex relaxations for quadratically constrained quadratic
  programming.
\newblock {\em Mathematical Programming}, 136(2):233--251.

\bibitem[Atamt{\"u}rk and G{\'o}mez, 2018]{atamturk2018strong}
Atamt{\"u}rk, A. and G{\'o}mez, A. (2018).
\newblock Strong formulations for quadratic optimization with {M-matrices} and
  indicator variables.
\newblock {\em Mathematical Programming}, 170(1):141--176.

\bibitem[Atamt\"urk and G{\'o}mez, 2019]{atamturk2019rank}
Atamt\"urk, A. and G{\'o}mez, A. (2019).
\newblock Rank-one convexification for sparse regression.
\newblock {\em Optimization Online}.
\newblock \url{http://www.optimization-online.org/DB\_HTML/2019/01/7050.html}.

\bibitem[Atamt\"urk et~al., 2021]{atamturk2018sparse}
Atamt\"urk, A., G{\'o}mez, A., and Han, S. (2021).
\newblock Sparse and smooth signal estimation: Convexification of {L0}
  formulations.
\newblock {\em Journal of Machine Learning Research}, 3:1--43.

\bibitem[Bacci et~al., 2019]{baccinew2019}
Bacci, T., Frangioni, A., Gentile, C., and Tavlaridis-Gyparakis, K. (2019).
\newblock New {MINLP} formulations for the unit commitment problems with
  ramping constraints.
\newblock {\em Optimization Online}.
\newblock \url{http://www.optimization-online.org/DB\_FILE/2019/10/7426.pdf}.

\bibitem[Belotti et~al., 2015]{belotti2015conic}
Belotti, P., G{\'o}ez, J.~C., P{\'o}lik, I., Ralphs, T.~K., and Terlaky, T.
  (2015).
\newblock A conic representation of the convex hull of disjunctive sets and
  conic cuts for integer second order cone optimization.
\newblock In {\em Numerical Analysis and Optimization}, pages 1--35. Springer.

\bibitem[Bertsimas et~al., 2020a]{bertsimas2020mixed}
Bertsimas, D., Cory-Wright, R., and Pauphilet, J. (2020a).
\newblock Mixed-projection conic optimization: A new paradigm for modeling rank
  constraints.
\newblock {\em arXiv preprint arXiv:2009.10395}.

\bibitem[Bertsimas and King, 2016]{bertsimas2016}
Bertsimas, D. and King, A. (2016).
\newblock {OR Forum} -- {An} algorithmic approach to linear regression.
\newblock {\em Operations Research}, 64(1):2--16.

\bibitem[Bertsimas et~al., 2016]{bertsimas2016best}
Bertsimas, D., King, A., and Mazumder, R. (2016).
\newblock Best subset selection via a modern optimization lens.
\newblock {\em The Annals of Statistics}, 44(2):813--852.

\bibitem[Bertsimas et~al., 2020b]{bertsimas2020sparse}
Bertsimas, D., Pauphilet, J., Van~Parys, B., et~al. (2020b).
\newblock Sparse regression: Scalable algorithms and empirical performance.
\newblock {\em Statistical Science}, 35(4):555--578.

\bibitem[Bertsimas and Van~Parys, 2017]{bertsimas2017sparse}
Bertsimas, D. and Van~Parys, B. (2017).
\newblock Sparse high-dimensional regression: Exact scalable algorithms and
  phase transitions.
\newblock {\em arXiv preprint arXiv:1709.10029}.

\bibitem[Bien et~al., 2013]{bien2013lasso}
Bien, J., Taylor, J., and Tibshirani, R. (2013).
\newblock A lasso for hierarchical interactions.
\newblock {\em Annals of Statistics}, 41(3):1111.

\bibitem[Bienstock and Michalka, 2014]{bienstock2014cutting}
Bienstock, D. and Michalka, A. (2014).
\newblock Cutting-planes for optimization of convex functions over nonconvex
  sets.
\newblock {\em SIAM Journal on Optimization}, 24(2):643--677.

\bibitem[Burer, 2009]{Burer2009}
Burer, S. (2009).
\newblock On the copositive representation of binary and continuous nonconvex
  quadratic programs.
\newblock {\em Mathematical Programming}, 120(2):479--495.

\bibitem[Burer and K{\i}l{\i}n{\c{c}}-Karzan, 2017]{burer2017convexify}
Burer, S. and K{\i}l{\i}n{\c{c}}-Karzan, F. (2017).
\newblock How to convexify the intersection of a second order cone and a
  nonconvex quadratic.
\newblock {\em Mathematical Programming}, 162(1-2):393--429.

\bibitem[Carrizosa et~al., 2020]{carrizosa2020linear}
Carrizosa, E., Mortensen, L., and Morales, D.~R. (2020).
\newblock On linear regression models with hierarchical categorical variables.
\newblock Technical report.

\bibitem[Ceria and Soares, 1999]{ceria1999convex}
Ceria, S. and Soares, J. (1999).
\newblock Convex programming for disjunctive convex optimization.
\newblock {\em Mathematical Programming}, 86:595--614.

\bibitem[Cozad et~al., 2014]{cozad2014learning}
Cozad, A., Sahinidis, N.~V., and Miller, D.~C. (2014).
\newblock Learning surrogate models for simulation-based optimization.
\newblock {\em AIChE Journal}, 60(6):2211--2227.

\bibitem[Cozad et~al., 2015]{cozad2015combined}
Cozad, A., Sahinidis, N.~V., and Miller, D.~C. (2015).
\newblock A combined first-principles and data-driven approach to model
  building.
\newblock {\em Computers \& Chemical Engineering}, 73:116--127.

\bibitem[Dedieu et~al., 2020]{dedieu2020learning}
Dedieu, A., Hazimeh, H., and Mazumder, R. (2020).
\newblock Learning sparse classifiers: Continuous and mixed integer
  optimization perspectives.
\newblock {\em arXiv preprint arXiv:2001.06471}.

\bibitem[Dheeru and Karra~Taniskidou, 2017]{Dua:2017}
Dheeru, D. and Karra~Taniskidou, E. (2017).
\newblock {UCI} machine learning repository.

\bibitem[Dong, 2019]{dong2019integer}
Dong, H. (2019).
\newblock On integer and {MPCC} representability of affine sparsity.
\newblock {\em Operations Research Letters}, 47(3):208--212.

\bibitem[Dong et~al., 2019]{dong2019structural}
Dong, H., Ahn, M., and Pang, J.-S. (2019).
\newblock Structural properties of affine sparsity constraints.
\newblock {\em Mathematical Programming}, 176(1-2):95--135.

\bibitem[Dong et~al., 2015]{dong2015regularization}
Dong, H., Chen, K., and Linderoth, J. (2015).
\newblock Regularization vs. relaxation: A conic optimization perspective of
  statistical variable selection.
\newblock {\em arXiv preprint arXiv:1510.06083}.

\bibitem[Dong and Linderoth, 2013]{dong2013valid}
Dong, H. and Linderoth, J. (2013).
\newblock On valid inequalities for quadratic programming with continuous
  variables and binary indicators.
\newblock In Goemans, M. and Correa, J., editors, {\em Integer Programming and
  Combinatorial Optimization}, pages 169--180, Berlin, Heidelberg. Springer.

\bibitem[Efron et~al., 2004]{efron2004least}
Efron, B., Hastie, T., Johnstone, I., and Tibshirani, R. (2004).
\newblock Least angle regression.
\newblock {\em The Annals of Statistics}, 32(2):407--499.

\bibitem[Fan and Li, 2001]{fan2001variable}
Fan, J. and Li, R. (2001).
\newblock Variable selection via nonconcave penalized likelihood and its oracle
  properties.
\newblock {\em Journal of the American Statistical Association},
  96(456):1348--1360.

\bibitem[Frangioni et~al., 2016]{frangioni2016approximated}
Frangioni, A., Furini, F., and Gentile, C. (2016).
\newblock Approximated perspective relaxations: a project and lift approach.
\newblock {\em Computational Optimization and Applications}, 63(3):705--735.

\bibitem[Frangioni and Gentile, 2006]{frangioni2006perspective}
Frangioni, A. and Gentile, C. (2006).
\newblock Perspective cuts for a class of convex 0--1 mixed integer programs.
\newblock {\em Mathematical Programming}, 106:225--236.

\bibitem[Frangioni and Gentile, 2007]{frangioni2007sdp}
Frangioni, A. and Gentile, C. (2007).
\newblock {SDP} diagonalizations and perspective cuts for a class of
  nonseparable {MIQP}.
\newblock {\em Operations Research Letters}, 35(2):181--185.

\bibitem[Frangioni and Gentile, 2009]{frangioni2009computational}
Frangioni, A. and Gentile, C. (2009).
\newblock A computational comparison of reformulations of the perspective
  relaxation: {SOCP} vs.\ cutting planes.
\newblock {\em Operations Research Letters}, 37(3):206--210.

\bibitem[Frangioni et~al., 2011]{frangioni2011projected}
Frangioni, A., Gentile, C., Grande, E., and Pacifici, A. (2011).
\newblock Projected perspective reformulations with applications in design
  problems.
\newblock {\em Operations Research}, 59(5):1225--1232.

\bibitem[Frangioni et~al., 2020]{frangioni2019decompositions}
Frangioni, A., Gentile, C., and Hungerford, J. (2020).
\newblock Decompositions of semidefinite matrices and the perspective
  reformulation of nonseparable quadratic programs.
\newblock {\em Mathematics of Operations Research}, 45(1):15--33.

\bibitem[G{\"u}nl{\"u}k and Linderoth, 2010]{gunluk2010perspective}
G{\"u}nl{\"u}k, O. and Linderoth, J. (2010).
\newblock Perspective reformulations of mixed integer nonlinear programs with
  indicator variables.
\newblock {\em Mathematical Programming}, 124:183--205.

\bibitem[Han et~al., 2020]{han20202x2}
Han, S., G{\'o}mez, A., and Atamt{\"u}rk, A. (2020).
\newblock 2x2 convexifications for convex quadratic optimization with indicator
  variables.
\newblock {\em arXiv preprint arXiv:2004.07448}.

\bibitem[Hardy, 1908]{hardy2018course}
Hardy, G.~H. (1908).
\newblock {\em Course of Pure Mathematics}.
\newblock Courier Dover Publications.

\bibitem[Hastie et~al., 2015]{hastie2015statistical}
Hastie, T., Tibshirani, R., and Wainwright, M. (2015).
\newblock {\em Statistical Learning with Sparsity: The Lasso and
  Generalizations}.
\newblock Monographs on statistics and applied probability, no. 143. Chapman
  and Hall/CRC.

\bibitem[Hazimeh and Mazumder, 2020]{hazimeh2019learning}
Hazimeh, H. and Mazumder, R. (2020).
\newblock Learning hierarchical interactions at scale: A convex optimization
  approach.
\newblock In Chiappa, S. and Calandra, R., editors, {\em Proceedings of the
  Twenty Third International Conference on Artificial Intelligence and
  Statistics}, volume 108 of {\em Proceedings of Machine Learning Research},
  pages 1833--1843. PMLR.

\bibitem[Hazimeh et~al., 2020]{hazimeh2020sparse}
Hazimeh, H., Mazumder, R., and Saab, A. (2020).
\newblock Sparse regression at scale: Branch-and-bound rooted in first-order
  optimization.
\newblock {\em arXiv preprint arXiv:2004.06152}.

\bibitem[Hijazi et~al., 2012]{hijazi2012mixed}
Hijazi, H., Bonami, P., Cornu{\'e}jols, G., and Ouorou, A. (2012).
\newblock Mixed-integer nonlinear programs featuring “on/off” constraints.
\newblock {\em Computational Optimization and Applications}, 52(2):537--558.

\bibitem[Huang et~al., 2012]{huang2012selective}
Huang, J., Breheny, P., and Ma, S. (2012).
\newblock A selective review of group selection in high-dimensional models.
\newblock {\em Statistical science: {A} Review Journal of the Institute of
  Mathematical Statistics}, 27(4).

\bibitem[Jeon et~al., 2017]{jeon2017quadratic}
Jeon, H., Linderoth, J., and Miller, A. (2017).
\newblock Quadratic cone cutting surfaces for quadratic programs with on--off
  constraints.
\newblock {\em Discrete Optimization}, 24:32--50.

\bibitem[K{\i}l{\i}n{\c{c}}-Karzan and Y{\i}ld{\i}z, 2014]{kilincc2014two}
K{\i}l{\i}n{\c{c}}-Karzan, F. and Y{\i}ld{\i}z, S. (2014).
\newblock Two-term disjunctions on the second-order cone.
\newblock In {\em International Conference on Integer Programming and
  Combinatorial Optimization}, pages 345--356. Springer.

\bibitem[K{\"u}{\c{c}}{\"u}kyavuz et~al., 2020]{KSMW20}
K{\"u}{\c{c}}{\"u}kyavuz, S., Shojaie, A., Manzour, H., and Wei, L. (2020).
\newblock Consistent second-order conic integer programming for learning
  {Bayesian} networks.
\newblock {\em arXiv preprint arXiv:2005.14346}.

\bibitem[Manzour et~al., 2021]{MKS19}
Manzour, H., K{\"u}{\c{c}}{\"u}kyavuz, S., Wu, H.-H., and Shojaie, A. (2021).
\newblock Integer programming for learning directed acyclic graphs from
  continuous data.
\newblock {\em {INFORMS} Journal on Optimization}, 3(1):46--73.

\bibitem[Miller, 2002]{miller2002subset}
Miller, A. (2002).
\newblock {\em Subset selection in regression}.
\newblock Chapman and Hall/CRC.

\bibitem[Modaresi et~al., 2016]{modaresi2016intersection}
Modaresi, S., K{\i}l{\i}n{\c{c}}, M.~R., and Vielma, J.~P. (2016).
\newblock Intersection cuts for nonlinear integer programming: Convexification
  techniques for structured sets.
\newblock {\em Mathematical Programming}, 155(1-2):575--611.

\bibitem[Natarajan, 1995]{Natarajan1995}
Natarajan, B.~K. (1995).
\newblock Sparse approximate solutions to linear systems.
\newblock {\em SIAM Journal on Computing}, 24(2):227--234.

\bibitem[Pilanci et~al., 2015]{pilanci2015sparse}
Pilanci, P., Wainwright, M.~J., and El~Ghaoui, L. (2015).
\newblock Sparse learning via boolean relaxations.
\newblock {\em Mathematical Programming}, 151:63--87.

\bibitem[Richard and Tawarmalani, 2010]{richard2010lifting}
Richard, J.-P.~P. and Tawarmalani, M. (2010).
\newblock Lifting inequalities: a framework for generating strong cuts for
  nonlinear programs.
\newblock {\em Mathematical Programming}, 121(1):61--104.

\bibitem[Sato et~al., 2016]{sato2016feature}
Sato, T., Takano, Y., Miyashiro, R., and Yoshise, A. (2016).
\newblock Feature subset selection for logistic regression via mixed integer
  optimization.
\newblock {\em Computational Optimization and Applications}, 64(3):865--880.

\bibitem[Stubbs and Mehrotra, 1999]{Stubbs1999}
Stubbs, R.~A. and Mehrotra, S. (1999).
\newblock A branch-and-cut method for 0-1 mixed convex programming.
\newblock {\em Mathematical Programming}, 86(3):515--532.

\bibitem[Tibshirani, 1996]{tibshirani1996regression}
Tibshirani, R. (1996).
\newblock Regression shrinkage and selection via the lasso.
\newblock {\em Journal of the Royal Statistical Society: Series B
  (Methodological)}, pages 267--288.

\bibitem[Vielma, 2019]{vielma2019small}
Vielma, J.~P. (2019).
\newblock Small and strong formulations for unions of convex sets from the
  {Cayley} embedding.
\newblock {\em Mathematical Programming}, 177(1-2):21--53.

\bibitem[Wang and K{\i}l{\i}n{\c{c}}-Karzan, 2020a]{wang2019generalized}
Wang, A.~L. and K{\i}l{\i}n{\c{c}}-Karzan, F. (2020a).
\newblock The generalized trust region subproblem: solution complexity and
  convex hull results.
\newblock {\em Forthcoming in Mathematical Programming}.

\bibitem[Wang and K{\i}l{\i}n{\c{c}}-Karzan, 2020b]{wang2019tightness}
Wang, A.~L. and K{\i}l{\i}n{\c{c}}-Karzan, F. (2020b).
\newblock On convex hulls of epigraphs of {QCQPs}.
\newblock In Bienstock, D. and Zambelli, G., editors, {\em Integer Programming
  and Combinatorial Optimization}, pages 419--432, Cham. Springer International
  Publishing.

\bibitem[Wang and K{\i}l{\i}n\c{c}-Karzan, 2021]{wang2021tightness}
Wang, A.~L. and K{\i}l{\i}n\c{c}-Karzan, F. (2021).
\newblock On the tightness of {SDP} relaxations of {QCQP}s.
\newblock {\em Forthcoming in Mathematical Programming}.

\bibitem[Wei et~al., 2020]{WGK20}
Wei, L., G{\'o}mez, A., and K{\"u}{\c{c}}{\"u}kyavuz, S. (2020).
\newblock On the convexification of constrained quadratic optimization problems
  with indicator variables.
\newblock In Bienstock, D. and Zambelli, G., editors, {\em Integer Programming
  and Combinatorial Optimization}, pages 433--447, Cham. Springer International
  Publishing.

\bibitem[Wu et~al., 2017]{wu2017quadratic}
Wu, B., Sun, X., Li, D., and Zheng, X. (2017).
\newblock Quadratic convex reformulations for semicontinuous quadratic
  programming.
\newblock {\em SIAM Journal on Optimization}, 27(3):1531--1553.

\bibitem[Xie and Deng, 2020]{xie2018ccp}
Xie, W. and Deng, X. (2020).
\newblock Scalable algorithms for the sparse ridge regression.
\newblock {\em SIAM Journal on Optimization}, 30(4):3359--3386.

\bibitem[Zhang, 2010]{zhang2010nearly}
Zhang, C.-H. (2010).
\newblock Nearly unbiased variable selection under minimax concave penalty.
\newblock {\em The Annals of Statistics}, 38:894--942.

\bibitem[Zheng et~al., 2014]{zheng2014improving}
Zheng, X., Sun, X., and Li, D. (2014).
\newblock Improving the performance of {MIQP} solvers for quadratic programs
  with cardinality and minimum threshold constraints: A semidefinite program
  approach.
\newblock {\em INFORMS Journal on Computing}, 26(4):690--703.

\end{thebibliography}

\revised{
\appendix

\section{The special case when $\conv(Q)$ is compact} \label{sec:app}
In this section, we give an extended formulation of $\clconv(Z_Q)$ based on an extended formulation of $\conv(Q_0)$. In particular, this alternative formulation is more favorable in cases when the number of facets of $\conv(Q)$ is polynomially bounded while $\conv(Q_0)$ has an exponential number of facets. We denote the facets of $\conv(Q)$ which do not contain zero by $\{F_\ell \}_{1 \leq \ell \leq k}$, and we write each  $F_\ell $ as $F_\ell  := \{z  \; | \; A_\ell  z  \leq b_\ell\}$. \citet{forbidden2015}  prove that $\conv(Q_0) = \conv\left( \bigcup_{1 \leq \ell \leq k} F_\ell  \right)$, and a natural extended formulation of $\conv(Q_0)$ is as follows:
\begin{subequations} \label{eq:ExtendedQ0}
\begin{align}
	&z =  \sum_{\ell \in [k]}  \hat z_\ell \label{eq:ExtendedQ0-1} \\
	&A_\ell  \hat z_\ell  \leq \lambda_\ell  b_\ell & \ell\in[k] \label{eq:ExtendedQ0-2} \\
	&\sum_{\ell\in[k]} \lambda_\ell  = 1, \; \lambda \ge \mathbf 0. \label{eq:ExtendedQ0-3}
\end{align}	
\end{subequations}

\begin{theorem}\label{thm:extended}
	\begin{align*}
	\clconv(Z_Q) = \proj_{(z,\beta,t)} \Big\{(z,\hat z, \lambda, \beta,t) \in \R_+^{(k+1)p + k}& \times \R^{p}\times \R \; | \; \eqref{eq:ExtendedQ0-1}\--\eqref{eq:ExtendedQ0-2},\;
	\sum_{\ell \in [k]}   \lambda_\ell   \leq 1,
	\nonumber\\
&t \geq f(\mathbf 1^\top  \beta),\;t \geq (\mathbf 1^\top \lambda)f\left(\frac{\mathbf 1^\top  \beta}{\mathbf 1^\top \lambda}\right)\Big\}.		
	\end{align*}
\end{theorem}

\begin{proof}
Let 
\begin{align*}
Z =\Big\{(z,\hat z, \lambda, \beta,t) \in \R_+^{(k+1)p + k} \times \R^{p}\times \R \; | \; & \eqref{eq:ExtendedQ0-1}\--\eqref{eq:ExtendedQ0-2},\;
		\sum_{\ell \in [k]}   \lambda_\ell   \leq 1, 
	t \geq f(\mathbf 1^\top  \beta), \\
	&t \geq (\mathbf 1^\top \lambda)f\left(\frac{\mathbf 1^\top  \beta}{\mathbf 1^\top \lambda}\right)\Big\}.	
\end{align*}	

First we show that $\proj_{(z,\beta,t)}(Z) \subseteq \clconv(Z_Q)$. Given any $(z,\hat z, \lambda, \beta,t)\in Z$, note that constraints  $z\in \conv(Q)$ and $t \geq f(\mathbf 1^\top \beta)$ defining  $ \clconv(Z_Q)$ are trivially satisfied. It remains to show that  
$t \geq (\pi^\top z)f\left(\frac{\mathbf 1^\top  \beta}{\pi^\top z}\right),\; \; \forall \pi \in \mathcal F$. For each $\pi \in \mathcal F$, we have
\begin{align*}
	\pi^\top z = \sum_{\ell\in[k]}   \pi^\top \hat z_\ell  = \sum_{\ell\in[k]}   \lambda_\ell  \pi^\top \left(\frac{\hat z_\ell }{\lambda_\ell }\right) \geq \sum_{\ell \in [k]}   \lambda_\ell,
\end{align*}
where the inequality follows from the fact that we must have either $\lambda_\ell  = 0$ and $\hat z_\ell  = \mathbf 0$ or $\lambda_\ell  > 0$ and $\frac{\hat z_\ell }{\lambda_\ell } \in F_\ell $ since each $F_\ell $ is a polytope contained in the half-space defined by inequality $\mathbf{1}^\top z\geq 1$. Thus,
from Lemma~\ref{lemma:perspective}, we have $t \geq (\mathbf 1^\top \lambda)f\left(\frac{\mathbf 1^\top  \beta}{\mathbf 1^\top \lambda}\right)\ge  (\pi^\top z)f\left(\frac{\mathbf 1^\top  \beta}{\pi^\top z}\right),\; \; \forall \pi \in \mathcal F$, hence  $\proj_{(z,\beta,t)}(Z) \subseteq \clconv(Z_Q)$.

Now, it remains to prove that $\clconv(Z_Q) \subseteq \proj_{(z,\beta,t)}(Z)$. 
For any $(z , \beta, t) \in \clconv(Z_Q)$ if $z \in \conv(Q_0)$, then there exist $\hat z_\ell $ and $\lambda_\ell $  that satisfy \eqref{eq:ExtendedQ0}
and $\mathbf 1^\top \lambda= 1$. Since $t \geq f(\mathbf 1^\top  \beta)$ for all $(z , \beta, t) \in \clconv(Z_Q)$, $(z, \beta, t) \in \proj_{(z,\beta,t)}(Z)$. If $z \in \conv(Q)\backslash \conv(Q_0)$, then, from Lemma \ref{lemma:Z0}, we can write $z$ as $z = \lambda_0 z_0$, $0 \leq \lambda_0 < 1$, and we may assume $z_0$ is on one of the facets of $\conv(Q_0)$ defined by $\hat \pi^\top z_0 = 1$ for some $\hat \pi \in \mathcal F$. By definition, $\forall \pi \in \mathcal F \; \;$ $\pi^\top z_0 \geq \hat \pi^\top z_0 = 1$ which implies $\lambda_0 = \hat \pi^\top z = \min_{\pi \in \mathcal F} \pi^\top z$. Since $z_0 \in \conv(Q_0)$, there exists $\hat z_\ell , \lambda_\ell $ such that $z_0 = \sum_{\ell\in[k]} \hat z_\ell $ and \eqref{eq:ExtendedQ0-2}--\eqref{eq:ExtendedQ0-3} hold.
Then 
\begin{align*}
	z = &\sum_{\ell \in [k]}  (\lambda_0 \hat z_\ell ) \\
	&A_\ell  (\lambda_0 \hat z_\ell ) \leq \lambda_0\lambda_\ell  b_\ell, & \ell \in [k] \\
	&\sum_{\ell \in [k]}  \lambda_0 \lambda_\ell  \leq 1, \lambda  \geq \mathbf 0,
\end{align*}
and we have $\sum_{\ell \in [k]}   \lambda_0 \lambda_\ell  = \lambda_0  = \min_{\pi \in \mathcal F} \pi^\top z$. Using Lemma~\ref{lemma:perspective}, we find that $t \geq (\pi^\top z)f\left(\frac{\mathbf 1^\top  \beta}{\pi^\top z}\right),\; \; \forall \pi \in \mathcal F$  implies that $t \geq (\sum_{\ell \in [k]}   \lambda_0 \lambda_\ell) f\left(\frac{\mathbf 1^\top  \beta}{\sum_{\ell \in [k]}   \lambda_0 \lambda_\ell}\right)\geq (\sum_{\ell \in [k]}    \lambda_\ell) f\left(\frac{\mathbf 1^\top  \beta}{\sum_{\ell \in [k]}  \lambda_\ell}\right)$. Hence, $(z, \beta, t) \in \proj_{(z,\beta,t)}(Z)$.

\end{proof}

}

\end{document}